\documentclass[arxiv]{myarticle}

\initbibli

\newglossaryentry{Balls}{name={\ensuremath{\bB^{[r]}}},description={set of $r$ balls},sort={Br}}
\newcommand{\Balls}[1]{\glsonlyfirst{Balls}{\bB^{\left[#1\right]}}}
\newcommand{\st}{\mathrm{sr}}
\newglossaryentry{ballsst}{name={\ensuremath{\bB^{[r]}_{\st}}},description={set of $r$ balls of the same type},sort={Brst}}
\newcommand{\Ballsst}[1]{\glsonlyfirst{ballsst}{\bB^{\left[#1\right]}_{\st}}}
\NewSymbol[index={\!\!\AnnN},sort=ACVFA]{ACVFann}{\ACVF}{$\ACVF$ with analytic structure}
\newcommand{\ACVFanneq}{\ACVFann<\eqop>}

\NewSymbol[sort=S]{Points}{\Ss}{points in balls}
\NewSymbol[sort=B]{balls}{\Bb}{reciprocal of $\Ss$}
\newglossaryentry{points}{name={\ensuremath{\_^{\Ss}}},description={points in balls},sort={Ss^}}
\newcommand{\points}[1]{\glsonlyfirst{points}{#1^{\Ss}}}
\newcommand{\Ordeq}{\trianglelefteqslant}
\newcommand{\Ord}{\triangleleft}
\newcommand{\tLeq}{\tL^{eq}}
\newcommand{\tNeq}{\tN^{\eqop}}

\newcommand{\grad}{\mathrm{grad}}

\newglossaryentry{Funform}{name={\ensuremath{\Psi_{\Delta,F}}},description={set of formulas generated by $\Delta$ and $F$},sort={\Psi\Delta F}}
\newcommand{\Funform}[2]{\glsonlyfirst{Funform}{\Psi_{#1,#2}}}

\NewSymbol[index={\D},exponant={\ltN}]{LltD}{\LL}{leading language for valued differential fields}
\NewSymbol[index={\D},exponant={\ltN,\eqop}]{LltDeq}{\LL}{leading language for valued differential fields with imaginaries}
\NewSymbol[index={\D,\lt}]{LltDlt}{\LL}{leading language for valued differential fields restricted to $\lt\cup\Valgp$}
\NewSymbol[index={\ltN}]{Dlt}{\D}{derivation on $\lt$}

\NewSymbol[function=true]{Prol}{\cP}{prolongation of the type space}

\newcommand{\stabop}{\mathrm{Stab}}
\NewSymbol[function=true]{stab}{\stabop}{stabilizer}

\NewSymbol[index={\sigma,P},exponant={\acN}]{LsPac}{\LL}{language of $\WFp$}
\NewSymbol[]{LGD}{\LGN[\D]}{Language of $\VDFG$}

\boolfalse{symb@first@tpprol}
\boolfalse{symb@glo@substr}
\boolfalse{symb@glo@ac}
\boolfalse{symb@first@ACVFann}
\boolfalse{symb@first@prol}
\boolfalse{symb@first@LDdiv}
\boolfalse{symb@first@Ldiv}
\boolfalse{symb@glo@LltDeq}
\boolfalse{symb@glo@NIP}
\boolfalse{symb@glo@val}
\boolfalse{symb@first@VDF}
\boolfalse{symb@glo@Wittf}

\title{Imaginaries and invariant types in existentially closed valued differential fields}
\author{Silvain Rideau\thanks{Partially supported by ValCoMo (ANR-13-BS01-0006)}}
\date{\today}

\begin{document}

\maketitle

\begin{abstract}
We answer three related open questions about the model theory of valued differential fields introduced by Scanlon. We show that they eliminate imaginaries in the geometric language introduced by Haskell, Hrushovski and Macpherson and that they have the invariant extension property. These two result follow from an abstract criterion for the density of definable types in enrichments of algebraically closed valued fields. Finally, we show that this theory is metastable.
\end{abstract}

In \cite{Sca-DValF}, Scanlon showed that equicharacteristic zero fields equipped with both a valuation and a contractive derivation (i.e. a derivation $\D$ such that for all $x$, $\val(\D(x))\geq \val(x)$) have a reasonably tractable model theory. Scanlon proved that the class of existentially-closed such differential valued differential fields, which we will denote $\VDF$, is elementary and he proved a quantifier elimination theorem for these structures. In this paper, we wish to investigate further their model theoretic properties.

A theory is said to eliminate imaginaries if for every definable set $D$ and every definable equivalence relation $E\subseteq D^{2}$, there exists an definable function $f$ such that $x E y$ if and only if $f(x) = f(y)$; in other words, a theory eliminates imaginaries if the category of definable sets is closed under quotients. In \cite{HasHruMac-ACVF}, Haskell, Hrushovski and Macpherson proved that, algebraically closed valued fields ($\ACVF$) do not eliminate imaginaries in any of the "usual" languages, but it suffices to add certain collections of quotients, the \emph{geometric sorts}, to obtain elimination of imaginaries. By analogy with the fact that differentially closed fields of characteristic zero ($\DCF[0]$) have no more imaginaries than algebraically closed fields ($\ACF$), it was conjectured that $\VDF$ also eliminates imaginaries in the geometric language with a symbol added for the derivation.

To prove their elimination results Haskell, Hrushovski and Macpherson developed the theory of \emph{metastability}, an attempt at formalising the idea that, if we ignore the value group, algebraically closed valued fields behave in a very stable-like way (cf. Section\,\ref{subsec:metastab} for precise definitions). Few examples of metastable theories are known, but $\VDF$ seemed like a promising candidate. Once again, the analogy with differentially closed fields is tempting. Among stable fields, algebraically closed fields are extremely well understood but are too tame (they are strongly minimal) for any of the more subtle behaviour of stability to appear. The theory $\DCF[0]$ of differentially closed fields in characteristic zero, on the other hand, is still quite tame (it is $\omega$-stable) but some pathologies begin to show and, by studying $\DCF[0]$, one gets a better understanding of stability. The theory $\VDF$ could play a similar role with respect to $\ACVF$: it is a more complicated in which to experiment with metastability.

Nevertheless, it was quickly realised that the metastability of $\VDF$ was an open question, one of the difficulties being to prove the invariant extension property. A theory has the invariant extension property if, as in stable theories, every type over an algebraically closed set $A$ has a "nice" global extension: an extension which is preserved under all automorphisms that fix $A$ (Definition\,\ref{def:inv ext}). In Theorem\,\ref{thm:VDF}, we solve these three questions, by showing that $\VDF$ eliminates imaginaries in the geometric language, has the invariant extension property and is metastable over its value group.

Following the general idea of \cite{Hru-EIACVF,Joh-EIACVF}, elimination of imaginaries relative to the geometric sorts is obtained as a consequence of the density of definable types over algebraically closed parameters and of computing the canonical basis of definable types in $\VDF$. This second part of the problem is tackled in \cite{RidSim-NIP}. Moreover, the invariant extension property is also a consequence of the density of definable types. One of the goals of this paper is, therefore, to prove density of definable types in $\VDF$: given any $A$-definable set $X$ in a model of $\VDF$, we find a type in $X$ which is definable over the algebraic closure of $A$.

Let $\Ldiv$ be the one sorted language for $\ACVF$ and $\LDdiv := \Ldiv\cup\{\D\}$ be the one sorted language for $\VDF$, where $\D$ is a symbol for the derivation. It follows from quantifier elimination in $\VDF$ that, to describe the $\LDdiv$-type of $x$ (denoted $p$), it suffices to give the $\Ldiv$-type of $\prol[\omega](x) := (\D^{n}(x))_{n<\omega}$ (denoted $\tpprol[\omega](p)$). Moreover, $p$ is consistent with $X$ if and only if $\tpprol[\omega](p)$ is consistent with $\prol[\omega](X)$. Note that $\tpprol[\omega](p)$ is the pushforward of $p$ by $\prol[\omega]$ restricted to $\Ldiv$ and, thus, $\tpprol[\omega](p)$ is definable if and only if $p$ is. Therefore it is enough to find a "generic" definable $\Ldiv$-type $q$ consistent with $\prol[\omega](X)$.

A definable $\Ldiv$-type is a consistent collection of definable $\Delta$-types where $\Delta$ is a finite set of $\Ldiv$-formulas and so we can ultimately reduce to finding, for any such finite $\Delta$, a "generic" definable $\Delta$-type consistent with some $\LDdiv(M)$-definable set (see Proposition\,\ref{prop:ind approx}). It follows that most of the preparatory work in the second part of this paper (Sections\,\ref{sec:type ball} to \ref{sec:reparam}) will focus on understanding $\Delta$-types for finite $\Delta$ in $\ACVF$.

An example of this convoluted back and forth between two languages $\Ldiv$ and $\LDdiv$ is underlying the proof of elimination of imaginaries in $\DCF[0]$; in that case the back and forth is between the language of rings and the language of differential rings, although, in the classical proof, it may not appear clearly. Take any set $X$ definable in $\DCF_{0}$, let $X_{n} := \prol[n](X)$ where $\prol[n](x) := (\D^{i}(x))_{0\leq i\leq n}$ and let $Y_{n}$ be the Zariski closure of $X_{n}$. Now, choose a consistent sequence $(p_{n})_{n<\omega}$ of $\ACF$-types such that $p_{n}$ has maximal Morley rank in $Y_{n}$. Because $\ACF$ is stable, all the $p_{n}$ are definable and, by elimination of imaginaries in $\ACF$, they already have canonical bases in the field itself. Then the complete type of points $x$ such that $\prol[n](x)\models p_{n}$ is also definable, it has a canonical basis of field points, and it is obviously consistent with $X$.

In $\ACVF$, we cannot use the Zariski closure because we also need to take into account valuative inequalities. But the balls in $\ACVF$ are combinatorially well-behaved and we can approximate sets definable in $\VDF$ by finite fibrations of balls over lower dimensional sets: cells in the $C$-minimal setting (see Section\,\ref{sec:approx}). Because $C$-minimality is really the core property of $\ACVF$ which we are using, the results presented here generalise naturally to any  $C$-minimal extension of $\ACVF$. We hope it might lead in the future to a proof that $\VDF$ with analytic structure has the invariant extension property and has no more imaginaries than $\ACVF$ with analytic structure (denoted $\ACVFann$) which is $C$-minimal. Note that we have no concrete idea of what those analytic imaginaries might be (see \cite{HasHruMac-An}).\medskip

The paper is organised as follows. The first part contains model theoretic considerations about $\VDF$. In Section\,\ref{sec:background}, we give some background and state Theorem\,\ref{thm:VDF}, our main new theorem about $\VDF$ whose proof uses most of what appears later in the paper. Section\,\ref{sec:prol} explores the properties of an analogue of prolongations on the type space. In Section\,\ref{sec:dcl VDF}, we study the definable and algebraic closures. Finally, in Section\,\ref{sec:metastab VDF}, we prove that metastability bases exist in $\VDF$.

Sections\,\ref{sec:type ball} to \ref{sec:approx} contain the proof of Theorem\,\ref{thm:dens def}, an abstract criterion for the density of definable types. In Section\,\ref{sec:type ball} we study certain "generic" $\Delta$-types, for $\Delta$ finite, in a $C$-minimal expansion $T$ of $\ACVF$ (see Definition\,\ref{def:gen type}). In Section\,\ref{sec:imp def}, we introduce the notion of \emph{quantifiable} types and we show that the previously defined "generic" types are quantifiable. In Section\,\ref{sec:reparam}, we consider definable families of functions into the value group, in $\ACVF$ and $\ACVFann$. We show that their germs are internal to the value group. In Section\,\ref{sec:approx}, we put everything together to prove Theorem\,\ref{thm:dens def}. Finally, in Section\,\ref{sec:EI IE}, we use this density result to give a criterion for elimination of imaginaries and the invariant extension property.

This paper also contains, as an appendix, improvements of known results on stable embeddedness in pairs of valued fields which are used in order to apply the results of \cite{RidSim-NIP}.

\part*{Model theory of valued differential fields}

\section{Background and main results}
\label{sec:background}

Whenever $X$ is a definable set (or a union of definable sets) and $A$ is a set of parameters, $X(A)$ will denote $X\cap A$. Usually in this notation there is an implicit definable closure, but we want to avoid that here because more often than not there will be multiple languages around. Similarly, if $\mathcal{S}$ is a set of definable sorts, we will write $\mathcal{S}(A)$ for $\bigcup_{S\in\mathcal{S}}S(A)$. Also, the symbol $\subset$ will denote strict inclusion.

For all the definitions concerning stability or the independence property, we refer the reader to \cite{Sim-Book}.

\subsection{Valued differential fields}

We will mostly study \emph{equicharacteristic zero} valued fields in the leading term language. It consists of three sorts $\K$, $\lt$ and $\Valgp$, maps $\ltf : \K \to \lt$ and $\vallt : \lt \to \Valgp$, the ordered group language on $\Valgp$ and the ring language on $\lt$ and $\K$. The group structure will be denoted multiplicatively on $\lt$ and additively on $\Valgp$.

A valued field $(K,\val)$ has a canonical $\Llt$-structure given by interpreting $\Valgp$ as its value group and $\lt$ as $(K/(1+\Mid))$ where $\Mid$ denotes the maximal ideal of  the valuation ring $\Val\subseteq K$. The map $\ltf$ is interpreted as the canonical projection $K \to \lt$. The function $\cdot$ on $\lt$ is interpreted as its (semi)-group structure. We have a short exact sequence $1\to\inv*{res}\to\inv*{lt}\to\Valgp\to 0$ where $\res := \Val/\Mid$ is the residue field. The function $+$ is interpreted as the function induced by the addition on the fibres $\lt[\gamma] := \vallt<-1>(\gamma)\cup\{0\}$ (and for all $x$, $y\in\lt$ such that $\vallt(x) < \vallt(y)$, we define $x + y = y + x := x$). Note that $(\lt[\gamma],+,\cdot)$ is a one dimensional $\res$-vector space and that $\lt[0] = \res$. Although $\vallt(0)$ is usually denoted $+\infty\neq\gamma$, we consider that $0$ lies in each $\lt[\gamma]$. In fact, it is the identity of the group $(\lt[\gamma],+)$.

The valued fields we consider are also endowed with a derivation $\D$ such that for all $x\in \K$, $\val(\D(x)) \geq \val(x)$. Such a derivation is called contractive. We denote $\LltD$ the language $\Llt$ enriched with two new symbols $\D : \K \to \K$ and $\Dlt : \lt \to \lt$. In a valued differential field with a contractive derivation, we interpret $\D$ as the derivation and $\Dlt$ as the function induced by $\D$ on each $\lt[\gamma]$. This function $\Dlt$ turns $\lt[\gamma]$ into a differential $\res$-vector space and for any $x$, $y\in\lt$, we have $\Dlt(x\cdot y) = \Dlt(x)\cdot y + x\cdot\Dlt(y)$. We denote by $\LltDlt$ the restriction of $\LltD$ to the sorts $\lt$ and $\Valgp$.

Let $\prol![n](x)$ denote $(x,\D(x),\ldots,\D^{n}(x))$ and $\prol[\omega](x)$ denote $(\D^{i}(x))_{i\in\Nn}$.

\begin{definition}($\D$-Henselian)
Let $(K,\val,\D)$ be a valued differential field. The field $K$ is said to be $\D$-Henselian if for all $P\in\Val(K)\{X\} := \Val(K)[X^{(i)}\mid i\in\Nn]$ and $a\in\Val(K)$, if $\val(P(a)) > 0$ and $\min_{i}\{\val(\frac{\partial}{\partial X^{(i)}}P(a))\} = 0$, then there exists $c\in\Val$ such that $P(c) = 0$ and $\resf(c) = \resf(a)$.
\end{definition}

\begin{definition}(Enough constants)
Let $(K,\val,\D)$ be a valued differential field. We say that $K$ has enough constants if $\val(\cst{K}) = \val(K)$ where $\cst{K} := \{x\in K\mid\D(x) = 0\}$ denotes the field of constants.
\end{definition}

Let $\Ldiv!$ be the one sorted language for valued fields. It consists of the ring language enriched with a predicate $x\Div y$ interpreted as $\val(x)\leq\val(y)$. Let $\LDdiv! := \Ldiv\cup\{\D\}$ and $\VDF!$ be the $\LDdiv$-theory of valued fields with a contractive derivation which are $\D$-Henselian with enough constants, such that the residue field is differentially closed of characteristic zero and the value group is divisible.

\begin{example}[Hahn VDF]
Let $(k,\D)\models \DCF[0]$ and $\Valgp$ be a divisible ordered Abelian group. We endow the Hahn field $K = k[[t^{\Gamma}]]$ of power series $\sum_{\gamma\in\Gamma} a_{\gamma} t^{\gamma}$ with well ordered support and coefficients in $k$, with the derivation $\D(\sum_{\gamma\in\Gamma} a_{\gamma} t^{\gamma}) := \sum_{\gamma\in\Gamma} \D(a_{\gamma}) t^{\gamma}$. Then $(K,\val,\D)\models\VDF$.
\end{example}

As in the case of $\ACVF$, $\VDF$ can also be considered in the one sorted, two sorted, the three sorted languages and the leading term language which are enrichments of the valued field versions with symbols for the derivation. Recall that an $\LL$-definable set $D$ in some $\LL$-theory $T$ is said to be stably embedded if, for all $M\models T$ and all $\LL(M)$-definable sets $X$, $X\cap D$ is $\LL(D(M))$-definable.

\begin{theorem}[known VDF](\cite{Sca-DValF,Sca-RelFrob})
\begin{thm@enum}
\item The theory $\VDF$ eliminates quantifiers and is complete in the one sorted language, the two sorted language, the three sorted language and the leading term language;
\item\label{str Valgp} The value group $\Valgp$ is stably embedded. It is a pure divisible ordered Abelian group;
\item\label{str res} The residue field $\res$ is stably embedded. It is a pure model of $\DCF[0]$;
\end{thm@enum}
\end{theorem}

\begin{proof}
By \cite[Theorem\,7.1]{Sca-DValF} we have field quantifier elimination in the three sorted language. The stable embeddedness and purity results for $\res$ and $\Valgp$ follow (see, for example, \cite[Remark\,A.10.2]{Rid-ADF}). Now, the theory induced on $\res$ and $\Valgp$ are, respectively, differentially closed fields and divisible ordered Abelian groups. Both of these theories eliminate quantifiers. Quantifier elimination in the three sorted language follows and so does qualifier elimination in the one sorted and two sorted languages.

As for the leading term structure, by \cite[Corollary\,5.8 and Theorem\,6.3]{Sca-RelFrob}, $\VDF$ eliminates quantifiers relative to $\lt$. Hence one can easily check that $\lt$ is stably embedded and it is a pure $\LltDlt$-structure. Quantifier elimination for $\VDF$ in the leading term language now follows from quantifier elimination for the structure induced on $\lt$ which we prove in the following lemma.
\end{proof}

\begin{lemma}[EQRV]
Let $T_{\lt}$ be the $\LltDlt$-theory of short exact sequences of Abelian groups $1\to \inv*{res}\to\inv*{lt}\to\Valgp\to 0$ such that $\res\models\DCF[0]$, for all $\gamma\in\Valgp$, $(\lt[\gamma],+,\cdot,\D)$ is a differential $\res$-vector space, for all $x$, $y\in\lt$, $\D(x\cdot y) = \D(x)\cdot y + x\cdot \D(y)$ and $\Valgp$ is a divisible ordered Abelian group. Then $T$ eliminates quantifiers.
\end{lemma}

\begin{proof}
If suffices to prove that for any $M$, $N\models T_{\lt}$ such that $N$ is $\card{M}^{+}$-saturated and for any partial isomorphism $f : M \to N$, there exists an isomorphism $g : M \to N$ extending $f$ defined on all of $M$.

Let $A$ be the domain of $f$. We construct the extension step by step. By quantifier elimination in divisible ordered Abelian groups, there exists $g : \Valgp(M) \to \Valgp(N)$ defined on all of $\Valgp(M)$ and extending $\restr{f}{\Valgp}$. It is easy to see that $f\cup g$ is a partial isomorphism. So we my assume that $\Valgp(A) = \Valgp(M)$. We may also assume that $A$ is closed under inverses: for any $a$, $b\in\lt(A)$, we define $g(a^{-1}\cdot b) := f(a)^{-1}\cdot f(b)$. Then $g \cup f$ is a partial isomorphism. Finally, By elimination of quantifiers in $\DCF[0]$, and saturation of $N$, $\restr{f}{\res}$can be extended to $h : \res(M) \to \res(N)$. Now define $g(\lambda\cdot a) = h(\lambda)\cdot f(a)$ for all $\lambda\in\res$ and $a\in A$. As $A$ is closed under inverse, this is well-defined and one can check that $f\cup g$ is indeed a partial isomorphism. So we may assume that $\res(M)\subseteq A$.

Let $a \in M$ and $\gamma = \vallt(a)$. If $a\nin A$, then $\lt[\gamma](A) = \emptyset$. Pick any $c\in\inv*{lt}[\gamma](M)$. We have $\D(c) = \lambda\cdot c$ for some $\lambda\in\res(M)$. We want to find $\mu\neq 0$ such that $\D(\mu\cdot c) = 0$, i.e. $\D(\mu)\cdot c + \mu\lambda\cdot c = 0$ and equivalently, $\D(\mu) + \lambda\mu = 0$. But this equation has a solution in $\res(M)$ as it is differentially closed. Thus, we may assume that $\D(c) = 0$. If there exist an $n\in\Nn_{>0}$ such that $c^{n}\in A$, let $n_{0}$ be the minimal such $n$, if such an $n$ does not exist, let $n_{0} = 0$. In both cases, let $b \in \lt[f(\gamma)](N)$ be such that $\D(b) = 0$ and $b^{n_{0}} = f(c^{n_{0}})$.

Now, for all $a\in \lt(A)$ and $n\in\Nn$, define $g(a\cdot c^{n}) = f(a)\cdot b^{n}$. It is easy to check that $g\cup f$ is a partial isomorphism. Applying this last construction repetitively, we obtain a morphism $g : M \to N$.
\end{proof}

\begin{remark}
If $M\models\VDF$, $\K(M)$ and $\cst{\K}(M)$ are algebraically closed, but $\K(M)$ is not differentially closed. In fact, the set $\{x\in \K(M)\mid \exists y\,\D(y) = x y\}$ defines the valuation ring.
\end{remark}

\subsection{Elimination of imaginaries}

Let us now recall some facts about elimination of imaginaries. A more thorough introduction can be found in \cite[Sections\,16.4 and 16.5]{Poi-MT}. An imaginary is a point in an interpretable set or equivalently a class of a definable equivalence relation. To every theory $T$ we can associate a theory $\eq{T}$ obtained by adding all the imaginaries. More precisely a new sort and a new function symbol are added for every $\emptyset$-interpretable set and they are interpreted respectively as the interpretable set itself and the canonical projection to the interpretable set. A model of $\eq{T}$ is usually denoted $\eq{M}$. We write $\dcleq$ and $\acleq$ to denote the definable and algebraic closure in $\eq{M}$.

We will also need to speak of the imaginaries internal to some $\star$-definable set. 

\begin{definition}[strict star-def]
Let $N$ an $\LL$-structure and $x$ a (potentially infinite) tuple of variables. Let $P$ be a set of $\LL$-formulas with variables $x$. The set $P(N) := \{m\in N^{x} \mid \forall \phi\in P,\,N\models\phi(m)\}$ is said to be $(\LL,x)$-definable. We say that an $(\LL,\star)$-definable set is strict $(\LL,\star)$-definable if the projection on any finite subset of $x$ is $\LL$-definable. If we do not want to specify $x$, (respectively $\LL$) we will simply say that a set is $(\LL,\star)$-definable (respectively $\star$-definable).
\end{definition}

If $X$ is an $(\LL(A),\star)$-definable set for some set of parameters $A$, $X$ can be considered as a structure with one predicate for each $\LL(A)$-definable subset of some Cartesian power of $X$. Then $\eq{X}$ will denote the structure imaginary structure on $X$, which can be seen as an $(\LL(A),\star)$-definable subset of $\eq{M}$. We might have to specify for which language is the induced structure is considered, in which case, we will write $\eq{X}_{\LL}$.

To every set $X$ definable (with parameters in some model $M$ of $T$), we can associate the set $\code{X}\subseteq \eq{M}$ which is the smallest definably closed set of parameters over which $X$ is defined. We usually call $\code{X}$ the code or canonical parameter of $X$.

Let $T$ be a theory in a language $\LL$ and $\Real$ a set of $\LL$-sorts. The theory $T$ eliminates imaginaries up to $\Real$ if every set $X$ definable with parameters is in fact definable over $\Real(\code{X})$; we say that $X$ is coded in $\Real$. If every $X$ is only definable over $\Real(\acleq(\code{X}))$, we say that $T$ weakly eliminates imaginaries. A theory eliminates imaginaries up to $\Real$ if and only if it weakly eliminates imaginaries up to $\Real$ and every finite set from the sorts $\Real$ is coded in $\Real$.

In \cite{HasHruMac-ACVF}, Haskell, Hrushovski and Macpherson introduced the geometric language $\LG$. It consists of a sort $\K$ for the valued field and, for all $n\in\Nn$, the sorts $\Latt[n] = \GL{n}(K)/\GL{n}(\Val)$ and the sorts $\Tor[n] = \GL{n}(K)/\GL{n,n}(\Val)$ where $\GL{n,n}(\Val)\leq\GL{n}(\Val)$ consists of the matrices which are congruent modulo the maximal ideal $\Mid$ to the matrix whose last columns contains only zeroes except for a one on the diagonal. The language also contains the ring language on $\K$ and the canonical projections onto $\Latt[n]$ and $\Tor[n]$. We will denote by $\Geom$ the sorts of the geometric language. Note that $\Latt[1]$ is exactly the value group and the canonical projection from $\inv*{K}$ onto $\Latt[1]$ is the valuation.

The main "raison d'être" of this geometric language is the following theorem:

\begin{theorem}(\cite[Theorem\,1.0.1]{HasHruMac-ACVF})
The theory $\ACVFG$ of algebraically closed valued fields in the geometric language eliminates imaginaries.
\end{theorem}

\subsection{Metastability}
\label{subsec:metastab}

Let $T$ be a theory, $M\models T$ be sufficiently saturated and $A\subseteq M$. The set $X$ is stable stably embedded if it is stably embedded and the $\LL(A)$-induced structure on $X$ is stable. We denote by $\St[A]<\LL>$ the structure whose sorts are the stable stably embedded sets which are $\LL(A)$-definable, equipped with their $\LL(A)$-induced structure. We will denote by $\indep_{C}^{\LL}$ forking independence in $\St[C]<\LL>$. When it is not necessary, we will not specify $\LL$.

\begin{definition}(Stable domination)
Let $M$ be an $\LL$-structure, $C\subseteq M$, $f$ an $(\LL(C),\star)$-definable map to $\St[C]$ and $p\in\TP(C)$. We say that $p$ is stably dominated via $f$ if for every $a\models p$ and $B\subseteq M$ such that $\St[C](\dcl(CB))\indep_{C} f(a)$,
\[\tp(B)[Cf(a)]\vdash\tp(B)[Ca].\]

We say that $p$ is stably dominated if it is stably dominated via some map $f$. It is then stably dominated via any map enumerating $\St[C](\dcl(Ca))$.
\end{definition}

\begin{definition}[inv ext](Invariant extension property)
Let $T$ be an $\LL$-theory that eliminates imaginaries, $A\subseteq M$ for some $M\models T$. We say that $T$ has the invariant extension property over $A$ if, for all $N\models T$, every type $p\in\TP(A)$ can be extended to an $\aut(N)[A]$-invariant type.

We say that $T$ has the invariant extension property if $T$ has invariant extensions over any $A = \acl(A)\subseteq M\models T$.
\end{definition}

\begin{definition}(Metastability)
Let $T$ be a theory and $\Valgp$ an $\emptyset$-definable stably embedded set. We say that $T$ is metastable over $\Valgp$ if:
\begin{thm@enum}
\item The theory $T$ has the invariant extension property.
\item For all $A\subseteq M$, there exists $C\subseteq M$ containing $A$ such that for all tuples $a\in M$, $\tp(a)[C\Valgp(\dcl(Ca))]$ is stably dominated. Such a $C$ is called a metastability basis.
\end{thm@enum}
\end{definition}

In \cite{HasHruMac-Book}, Haskell, Hrushovski and Macpherson showed that $\ACVF$ is metastable over its valued group and that  maximally complete fields are metastability bases. Recall that a valued field $(K,\val)$ is maximally complete if every chain of ball contains a point or equivalently every pseudo-Cauchy sequence from $K$ (a sequence $(x_{\alpha})_{\alpha\in\kappa}$ such that for all $\alpha<\beta<\gamma$, $\val(x_{\gamma}-x_{\beta}) > \val(x_{\beta} -x_{\alpha})$) have a pseudo-limit in $K$ (a point  $a\in K$ such that for all $\alpha<\beta$, $\val(a-x_{\beta}) > \val(a-x_{\alpha})$).

To finish this section, let us introduce two other kinds of types which coincide with stably dominated types in $\NIP$ metastable theories.

\begin{definition}(Generic stability)
Let $M$ be some $\NIP$ $\LL$-structure and $p\in\TP(M)$. The type $p$ is said to be generically stable if it is $\LL(M)$-definable and finitely satisfiable in some (small) $N\subsel M$.
\end{definition}

\begin{definition}(Orthogonality to $\Valgp$)
Let $M\models T$ be sufficiently saturated, $C\subseteq M$, $\Valgp$ be an $\LL$-definable set and $p\in\TP(M)$ be an $\aut(M)[C]$-invariant type. The type $p$ is said to be orthogonal to $\Valgp$ if for all $B\subseteq M$ containing $A$ and $a\models \tprestr{p}{B}$, $\Valgp(\dcl(Ba)) = \Valgp(\dcl(B))$.
\end{definition}

\subsection{New results about \texorpdfstring{$\VDF$}{VDF}}

Let $\LGD$ be the language $\LG$ enriched with a symbol for the derivation $\D : \K\to \K$ and let $\VDFG!$ be the $\LGD$-theory of models of $\VDF$. The goals of this paper is to prove the following:

\begin{theorem}[VDF]
The theory $\VDFG$ eliminates imaginaries, has the invariant extension property and is metastable. Moreover, over algebraically closed sets of parameters, definable types are dense.
\end{theorem}

By density of definable types, we mean that every definable set $X$ is consistent with a global $\LGD(\acleq(\code{X}))$-definable type $p$.

\begin{proof}
The density of definable types is proved in Corollary\,\ref{cor:VDF dens def}. Elimination of imaginaries and the invariant extension property are proved in Corollary\,\ref{cor:VDF EI IE}. Finally, the existence of metastability bases is proved in Corollary\,\ref{cor:VDF meta}
\end{proof}

At the very end of \cite{HasHruMac-Book}, an incorrect proof of the metastability of $\VDF$ (in particular, of the invariant extension property) is sketched. Because it overlooks major difficulties inherent to the proof of the invariant extension property, there can be no easy way to fix this proof and new techniques had to be developed.

\section{Prolongation of the type space}
\label{sec:prol}

The goal of this section is to study the relation between types in $\VDF$ and types in $\ACVF$. This construction plays a fundamental role in the rest of this paper. However, in the proof of Theorem\,\ref{thm:dens def}, it appears in a more abstract setting.

For all $x\in\K$ or $x\in\lt$, let $\prol[\omega](x)$ denote $(\Dlt<n>(x))_{n\in\Nn}$. If $x\in\Valgp$, let $\prol[\omega](x)$ denote $(x)_{n\in\Nn}$. If $x$ is a tuple of variables, we denote by $x_{\infty}$ the tuple $(x^{(i)})_{i\in\Nn}$ where each $x^{(i)}$ is sorted like $x$. Let $M\models\VDF$ be sufficiently saturated and $A\substr M$ be a substructure. We write $\TP[x]<\LL>(A)$ for the space of complete $\LL$-types over $A$ in the variable $x$.

\begin{definition}
Let $A\subseteq\K\cup\Valgp\cup\lt$. We define $\tpprol![\omega]:\TP[x]<\LltD>(A) \to \TP[x_{\infty}]<\Llt>(A)$ to be the map which sends a complete type $p$ to the complete type
\[\tpprol[\omega](p) := \{\phi(x_{\infty},a)\mid \phi\text{ is an $\Llt$-formula and }\phi(\prol[\omega](x),a)\in p\}.\]
\end{definition}

\begin{proposition}
The function $\tpprol[\omega]$ is a homeomorphism onto its image (which is closed).
\end{proposition}

\begin{proof}
As $\TP[x]<\LltD>(A)$ is compact and $\TP[x_{\infty}]<\Llt>(A)$ is Hausdorff, it suffices to show that $\tpprol[\omega]$ is continuous and injective. Let us first show continuity. Let $U = \cl{\phi(x_{\infty},a)} \subseteq  \TP[x_{\infty}]<\Llt>(A)$, then $\tpprol[\omega]<-1>(U) = \cl{\phi(\prol[\omega](x),a)}\subseteq \TP[x]<\LltD>(A)$. As for $\tpprol[\omega]$ being injective, let $p$ and $q\in \TP[x]<\LltD>(A)$ and let $\phi(x,a)$ be an $\LltD$-formula in $p \sminus q$. By quantifier elimination, we can assume that $\phi$ is of the form $\theta(\prol[\omega](x),a)$ for some $\Llt$-formula $\theta$. Then $\theta(x_{\infty},a)\in\tpprol[\omega](p)\sminus\tpprol[\omega](q)$.
\end{proof}

We will now look at how $\tpprol[\omega]$ and its inverse behave with respect to various properties of types. Transferring certain properties actually presents real challenges: proving Propositions\,\ref{prop:prol def} and \ref{prop:prol inv} required the development of \cite{RidSim-NIP}. Note that in \cite{RidSim-NIP} the variables of the type $p$ are in $\K$, the same proof applies if the variables are in $\K$, $\lt$ and $\Valgp$ (we have to use elimination of quantifiers in $\LltD$ instead).

\begin{proposition}[prol def](\cite[Corollary\,3.3]{RidSim-NIP})
Let $p\in\TP<\LltD>(M)$. Assume $A = \acleq[\LltD](A)$. The following are equivalent:
\begin{thm@enum}
\item $p$ is $\LltDeq(A)$-definable; 
\item $\tpprol[\omega](p)$ is $\LG(\Geom(A))$-definable; 
\item $p$ is $\LGD(\Geom(A))$-definable.
\end{thm@enum}
\end{proposition}

\begin{proposition}[prol inv](\cite[Corollary\,3.5]{RidSim-NIP})
Let $p\in\TP<\LltD>(M)$. Assume $A = \acleq[\LltD](A)$. The following are equivalent:
\begin{thm@enum}
\item $p$ is $\aut[\LltDeq](M)[A]$-invariant;
\item $\tpprol[\omega](p)$ is $\aut[\LG](M)[\Geom(A)]$-invariant;
\item $p$ is $\aut[\LGD](M)[\Geom(A)]$-invariant.
\end{thm@enum}
\end{proposition}

\begin{proposition}[prol inv st dom]
Let $p\in\Prol[x](A)$, $f$ be an $(\LG(A),\star)$-definable map defined on $p$ and let $D$ be the image of $f$. Assume that $A\subseteq\K\cup\Valgp\cup\lt$, $p$ is stably dominated via $f$ and that $\eq{D}_{\Llt}(\acleq[\LltD](A)) = \eq{D}_{\Llt}(\acleq[\Llt](A))$. Then $\tpprol[\omega]<-1>(p)$ is also stably dominated (via $f\comp\prol[\omega]$).
\end{proposition}

\begin{proof}
We will need the following result:
\begin{claim}
Let $D$ be $\LG(M)$-definable. If $D$ is stable and stably embedded in $\ACVF$, then it is also stable and stably embedded in $\VDF$.
\end{claim}

\begin{proof}
It follows from \cite[Lemma\,2.6.2 and Remark\,2.6.3]{HasHruMac-ACVF} that $D\subseteq\dcl[\LG](E\cup\res)$ for some finite $E\subseteq D$. Because $\res$ also eliminates imaginaries, is stable and stably embedded in $\VDF$, it immediately follows that $D$ is stably embedded and stable in $\VDF$ too.
\end{proof}

Now let $c\models \tpprol[\omega]<-1>(p)$ and $B\subseteq\K$ be such that
\[\St[A]<\LltD>(\dcl[\LltD](AB)) \indep_{A}^{\LltD} f(\prol[\omega](c)).\]
By hypothesis, $\eq{D}_{\Llt}(\acleq[\LltD](A)) = \eq{D}_{\Llt}(\acleq[\Llt](A))$, and hence
\[\St[A]<\Llt>(\dcl[\Llt](A\prol[\omega](B))) \indep_{A}^{\Llt} f(\prol[\omega](c)).\]
Since $\prol[\omega](c)\models p$ and $p$ is stably dominated via $f$,
\[\begin{eqn}
\tp[\LltD](B)[A f(\prol[\omega](c))]&\vdash& \tp[\Llt](\prol[\omega](B))[A f(\prol[\omega](c))]\\
&\vdash&\tp[\Llt](\prol[\omega](B))[A\prol[\omega](c)]\\
&\vdash&\tp[\LltD](B)[Ac].
\end{eqn}\]
The last implication comes from the fact that $\tpprol[\omega]$ is one to one on the space of types.
\end{proof}

\begin{proposition}
Let $M\models\VDF$ be sufficiently saturated and homogeneous, $A\subseteq M$ and $p\in\TP[x]<\LGD>(M)$ be $\aut[\LGD](M)[A]$-invariant. The following are equivalent:
\begin{multicols}{2}
\begin{thm@enum}
\item $p$ is stably dominated;
\item $p$ is generically stable;
\item $p$ is orthogonal to $\Valgp$;
\item $\tpprol[\omega](p)$ is stably dominated;
\item $\tpprol[\omega](p)$ is generically stable;
\item $\tpprol[\omega](p)$ is orthogonal to $\Valgp$;
\end{thm@enum}
\end{multicols}
\end{proposition}

\begin{proof}
Since $\tpprol[\omega](p)$ is an $\ACVF$-type. The equivalence of (iv), (v) and (vi) is proved in \cite[Proposition\,2.8.1]{HruLoe}. Actually, the implications (i) $\imp$ (ii) $\imp$ (iii) hold in any $\NIP$ theory where $\Valgp$ is ordered. 

We proved in Proposition\,\ref{prop:prol inv st dom} that (iv) implies (i). Let us now prove that (iii) implies (iv). By Proposition\,\ref{prop:prol def}, $\tpprol[\omega](p)$ is $\aut[\LG](M)[C]$-invariant for some $C\subseteq M$. We may assume that $C\models\VDF$ and is maximally complete. Let $c\models\tprestr{p}{C}$. As $p$ is orthogonal to $\Valgp$, we have $\Valgp(C) \subseteq \Valgp(\dcl[\Llt](Cc)) \subseteq \Valgp(\dcl[\LltD](Cc)) = \Valgp(C)$. As $C$ is maximally complete, we have $\tp[\Ldiv](\prol[\omega](c))[C\Valgp(\dcl[\Llt](Cc))]$ is stably dominated (see \cite[Theorem\,12.18.(ii)]{HasHruMac-Book}). But
\[\tp[\Ldiv](\prol[\omega](c))[C\Valgp(\dcl[\Llt](Cc))] = \tp[\Ldiv](\prol[\omega](c))[C] = \tprestr{\tpprol[\omega](p)}{C}\] and hence $\tpprol[\omega](p)$ is also stably dominated.
\end{proof}

\section{Definable and algebraic closure in \texorpdfstring{$\VDF$}{VDF}}
\label{sec:dcl VDF}

In this section, we investigate the definable and algebraic closures in $\VDF$. We show that they are not as simple as one might hope. In $\DCF[0]$, the definable closure of $a$ is exactly the field generated by $\prol[\omega](a)$. In $\VDF$, we have, at least, to take in account the Henselianisation, but we show that the definable closure (in the field sort) of a new field element $a$ can be even larger than the Henselianisation of the field generated by $\prol[\omega](a)$. This fact was already known to Ehud Hrushovski and Thomas Scanlon but was never written down. However, we also show that the $\Valgp$, $\res$ and $\lt$ points of the definable closure (respectively algebraic closure) is exactly what one would expect: the $\ACVF$ definable closure (respectively algebraic closure) of the differential structure generated by the parameters.

We will, again, be working in the leading term language and all the sets of parameters that appear in this section will be living in the sorts $\K\cup\Valgp\cup\lt$. We denote by $\dalrg{A}$ (respectively $\dalfield{A}$) the $\LltD$-structure generated by $A$ (respectively the closure of $A$ under both $\LltD$-terms and inverses).

\begin{proposition}
Let $M\models\VDF$ be sufficiently saturated. For all $C\subseteq M$, there exists $A\subseteq M$, such that $C\subseteq A$ and $\K(\dcl[\Llt](\gen{\D}{A})) = \hens{\K(\dalfield{A})}\subset\K(\dcl[\LDdiv](A))$. In fact, there exists $a\in \K(\dcl[\LDdiv](A))$ which is transcendental over $\gen{\D}{A}$. In particular, we also have $a\nin \acl[\Llt](\gen{\D}{A})$.
\end{proposition}

We show that certain differential equations which have infinitely many solutions in differentially closed fields have only one solution in models of $\VDF$. We then show that this unique solution is not algebraic over the parameters.

\begin{proof} Let $P(X)\in\Val(M)\{X\}$, $a\in\Val(M)$ and $\epsilon\in\Mid(M)$. Let $Q_{a}(x) = x - a + \epsilon P(x)$. Then $Q_{a}$ has a unique zero in $M$. Indeed $\val(Q_{a}(a)) > 0$, $\val(\frac{\partial Q_{a}}{\partial X^{(0)}}(a)) = \val(1) = 0$ and $\val(\frac{\partial Q_{a}}{\partial X^{(i)}}(a)) = \val(\epsilon) + \val(\frac{\partial P}{\partial X_{i}}(a)) > 0$, hence $\sigma$-Henselianity applies and $Q_a$ has at least one zero.

If $Q_{a}(x) = Q_{a}(y) = 0$, then $\resf(x) = \resf(a) = \resf(y)$. Let $\eta := x-y$, we have \[Q_{a}(y) = x + \eta - a + \epsilon P(x+\eta) = x + \eta - \epsilon(\sum_{I}P_{I}(a)\eta^{I}) = \eta + \epsilon(\sum_{\card{I} > 0}P_{I}(a)\eta^{I}).\]
But, if $\eta\neq 0$, $\val(\epsilon P_{I}(a)\eta^{I}) > \card{I}\val(\eta) \geq \val(\eta)$ and hence $\val(Q_{a}(y)) = \val(\eta) \neq \infty$, a contradiction. Hence the equation $Q_a(x) = 0$ has a unique solution in $M$.

Let us now show that, if $a$ and $P$ are chosen correctly, the solution to this equation is not algebraic. We may assume that $C = \dcl[\Llt](\K(C))$. Let $k$ be a differential field, $\widetilde{a}\in k$ be differentially transcendental. Let us equip $k[[\epsilon]]$ with the usual contractive derivation (cf. Example\,\ref{ex:Hahn VDF}). We embed $k[[\epsilon]]$ in $M$ so that $k$ and $\res(C)$ are independent and $\K(C)(\epsilon)$ is a transcendental ramified extension of $\K(C)$. To avoid any confusion, let us denote by $a$ the image of $\widetilde{a}$ by the embedding of $k$ into $k[[\epsilon]]$ and into $M$. One can check that for all $n\in\Nn$, $\resf(\K(C)(\epsilon,\prol[n](a))) = \res(C)(\prol[n](\widetilde{a}))$.

Let us now try to solve $x-a - \epsilon \D(x) = 0$ in $k[[\epsilon]]$. Let $x = \sum x_{i}\epsilon^{i}$ where $x_{i} \in k$, the equation can then be rewritten as:
\[\sum x_{i}\epsilon^{i} = a\epsilon^{0} + \sum \D(x_{i})\epsilon^{i+1},\]
Hence $x_{0} = a$ and $x_{i+1} = \D(x_{i}) = \D^{i+1}(a)$. If $x\in\alg{\dalfield{C,a,\epsilon}}$ then for some $n\in\Nn$, we must have $x\in\alg{\K(C)(\prol[n](a),\epsilon)}$. Any automorphism of $\sigma:k\cup\res(C)$ fixing $\res(C)$ can be lifted into an automorphism of $k[[\epsilon]]\cup C$ fixing $C$ and sending $\sum x_{i}\epsilon^{i}\in k[[\epsilon]]$ to $\sum \sigma(x_{i})\epsilon^{i}$. Because $\D^{n+1}(\widetilde{a})$ is transcendental over $\res(C)(\prol[n](\widetilde{a}))$, it follows that $x$ has an infinite orbit over $A = \K(C)(\prol[n](a),\epsilon)$. Therefore $x \in\dcl[\LltD](A) \sminus \alg{A}$.
\end{proof}

Let us now consider what happens for $\Valgp$ and $\res$.

\begin{proposition}[dcl acl G res]
Let $M\models\VDF$ and $A\subseteq M$, then
\[\Valgp(\dcl[\LltD](A)) = \Valgp(\acl[\LltD](A)) = \DIV{\Valgp(\dalfield{A})},\]
\[\res(\dcl[\LltD](A)) = \res(\dalfield{A})\text{ and }\res(\acl[\LltD](A)) = \alg{\res(\dalfield{A})}.\]
\end{proposition}

\begin{proof}
Let us first show that $\Gamma(\acl[\LltD](A)) = \DIV{\val(\dalfield{A})}$. By quantifier elimination in the leading term language, any formula with variables in $\Valgp$ and parameters in $A$ is of the form $\phi(x,a)$ where $a\in\Valgp(\dalrg{A})$ (note that this is a stronger result than stable embeddedness of $\Valgp$ as we have strong control over the new parameters). In particular, any $\gamma\in\Valgp(M)$ algebraic over $A$ is algebraic over $\Valgp(\dalrg{A})$, in $\Valgp$ which is a pure divisible ordered Abelian group. It follows immediately that $\gamma\in \DIV{\Valgp(\dalfield{A})}$. Finally, as $\DIV{\val(\dalfield{A})}$ is rigid over $\val(\dalfield{A}) \subseteq\Valgp(\dcl[\LltD](A))$ the equality $\Valgp(\dcl[\LltD](A)) = \Valgp(\acl[\LltD](A))$ also holds.

As for the results concerning $\res$, they are proved similarly. Indeed, any formula with variables in $\res$ and parameters in $A$ is of the form $\phi(x,a)$ where $a\in\res(\dalfield{A})$ is a tuple. The proof of this fact requires a little more work than for $\Valgp$ because formulas of the form $\sum_{i\in I} a_{i} x^{i} = 0$ where $a_{i}\in\lt(\dalfield{A})$ are not immediately seen to be of the right form. But we may assume that all $a_{i}$ have the same valuation (as only the monomials with minimal valuation are relevant to this equation). Hence, this formula is equivalent to $\sum_{i\in I} a_{i}a_{i_{0}}^{-1} x^{i} = 0$ which is of the right form.

The results now follow from the fact that in $\DCF[0]$ the definable closure is just the differential field generated by the parameters and the algebraic closure is its field theoretic algebraic closure.
\end{proof}

\begin{proposition}[acl RV]
For all $M\models\VDF$ and $A \subseteq M$, $\lt(\dcl[\LltD](A)) = \lt(\dalfield{A})$ and $\lt(\acl[\LltD](A)) = \lt(\acl[\Llt](\dalrg{A}))$.
\end{proposition}

\begin{proof} Let $\tA := (\lt\cup\Valgp)(\dalfield{A})$. By quantifier elimination for $\VDF$ in the leading term language, any formula with variables in $\lt$  and parameters in $A$ is of the form $\phi(x,a)$ where $a\in\tA$ is a tuple. In particular, $\lt(\dcl[\LltD](A)) = \lt(\dcl[\LltD](\tA))$ and $\lt(\acl[\LltD](A)) = \lt(\acl[\LltD](\tA))$.

\begin{claim}[fibre out div]
For all $\gamma\in\Valgp\sminus\DIV{\vallt(\lt(\tA))}$, $\lt[\gamma](\acl[\LltD](\tA)) = \emptyset$.
\end{claim}

\begin{proof}
Pick any $(d_{i})_{i\in\Nn}\in \res$ such that $d_{m n}^{n} = d_{m}$ and $\D(d_{m}) = 0$ for all $m$ and $n\in\Nn$. Let us write $\Valgp = (\DIV{\Valgp(\tA)})\bigoplus_{i\in I} \Qq\gamma_{i}$ where one of the $\gamma_{i}$ is $\gamma$. Define a group morphism $\sigma : \Valgp \to \res(M)$ sending all of $\DIV{\Valgp(\tA)}$ and all $\gamma_{i}\neq\gamma$ to $0$ and $p/q\cdot \gamma$ to $d_{q}^{p}$. For all $x\in\lt$, we now define $\tau(x) = \sigma(\vallt(x))\cdot x$. It is easy to check that $\tau$ is an $\LltDlt$-automorphism of $\lt$ and that $\tau$ fixes $\tA$.

On the fibre $\lt[\gamma]$, $\tau$ sends $x$ to $d_{1}\cdot x$. It immediately follows that, because we have infinitely many choices for $d_{1}$ (as $\cst{\res}$ is algebraically closed), the $\aut[\LltDlt](\lt)[\tA]$-orbit of $x$ is infinite. Thus $\lt[\gamma](\dcl[\LltD](\tA)) = \lt[\gamma](\acl[\LltD](\tA)) = \emptyset$.
\end{proof}

\begin{claim}[fibre div]
For all $\gamma\in\DIV{\vallt(\lt(\tA))}\sminus\vallt(\lt(\tA))$, $\lt[\gamma](\dcl[\LltD](\tA)) = \emptyset$ and $\lt[\gamma](\acl[\Llt](\tA)) \neq \emptyset$.
\end{claim}

\begin{proof}
Let $n$ be minimal such that $\delta = \gamma^{n}\in\vallt(\lt(\tA))$. Taking $d_{i}$ as above, with $d_{1} = 1$, and defining $\sigma$ such that $\sigma(p/q \cdot\gamma) = d_{q}^{p}$, we obtain an $\LltD$-automorphism $\tau$ which fixes $\tA$ and acts on $\lt[\gamma]$ by multiplying by $d_{n}$. As there are $n$ choices for $d_{n}$, we obtain that $\lt[\gamma](\dcl[\LltD](\tA)) = \emptyset$.

Now, let us show that $\lt[\gamma](\acl[\Llt](\tA)) \neq \emptyset$. Let $c\in \lt[\gamma]$. We have $\vallt(c^{n}) = \gamma^{n}$ and there exists $\lambda\in\res$ such that $\lambda\cdot c^{n}\in \tA$. Let $\mu\in\res$ be such that $\mu^{n} = \lambda$ and let $a = \mu\cdot c$. Then $a^{n} = \lambda\cdot c^{n}\in \tA$. As, the kernel of $x\mapsto x^{n}$ is finite, we have $a \in \acl[\Llt](\tA)$.
\end{proof}

Let $c\in\lt(\dcl[\LltD](\tA))$. By Claims\,\ref{claim:fibre div} and \ref{claim:fibre out div}, $\vallt(c) = \vallt(a)$ for some $a\in\tA$. It follows that $c\cdot a^{-1}\in \res(\dcl[\LltD](\tA))$, which, by Proposition\,\ref{prop:dcl acl G res} is equal to $\res(\tA)$. Hence $c = (c\cdot a^{-1})\cdot a \in \lt(\tA)$. If $c\in\lt(\acl[\LltD](\tA))$, by Claim\,\ref{claim:fibre out div} and \ref{claim:fibre div}, $\vallt(c) = \vallt(a)$ for some $a\in\acl[\Llt](\tA)$. Then $c\cdot a^{-1}\in \res(\acl[\LltD](\tA)) = \res(\acl[\Llt](\tA))$.
\end{proof}

Concerning the definable closure and algebraic closure in the sort $\K$, although the situation is not ideal, we nevertheless have some control over it:

\begin{corollary}[acl imm]
Let $M\models\VDF$ and $A\subseteq \K(M)$, then $\K(\acl[\LltD](A))$ is an immediate extension of $\alg{\K(\dalfield{A})}$.
\end{corollary}

\begin{proof}
By Proposition\,\ref{prop:dcl acl G res}, we have that $\val(\K(\acl[\LltD](A)))\subseteq \val(\alg{\K(\dalfield{A})})$ and that $\resf(\K(\acl[\LltD](A)))\subseteq \resf(\alg{\K(\dalfield{A})})$.
\end{proof}

\begin{corollary}[dcl imm]
Let $M\models\VDF$ and $A\subseteq \K(M)$ then $\K(\dcl[\LltD](A))$ is an immediate extension of $\K(\dalfield{A})$.
\end{corollary}

\begin{proof}
Let $L := \K(\dcl[\LltD](A))$ and $F := \hens{\K(\dalfield{A})}$. By Proposition\,\ref{prop:dcl acl G res}, we have that $\resf(L)\subseteq \resf(F)$ and that $\val(L)\subseteq \DIV{\val(F)}$. Let $c\in L$. We already know that $\val(c) \in\DIV{\val(F)}$. Let $n$ be minimal such that $n\cdot\val(c) = \val(a)$ for some $a\in F$. Let us show that $n=1$. We have $\resf(ac^{-n}) \in\resf(L) = \resf(F)$, so we can find $u\in F$ such that $\resf(ac^{-n}) = \resf(u)$. As $L$ must be Henselian (indeed $\hens{L} = \dcl[\Ldiv](L) = L$), we can find $v\in L$ such that $v^{n} = ac^{-n}u^{-1}$, i.e. $(c v)^{n} = a u^{-1} \in F$. Hence we may assume that $c^{n}$ itself is in $F$.

Derivations have a unique extension to algebraic extensions and, as $F$ is Henselian, the valuation also has a unique extension to the algebraic closure. It follows that any algebraic conjugate of $c$ is also an $\LDdiv$-conjugate of $c$. As $\K(M)$ is algebraically closed, it contains non trivial $n$-th roots of the unit. It follows that we must have $n = 1$.

We have just proved that $\K(\dcl[\LltD](A))$ is an immediate extension of $\hens{\K(\dalfield{A})}$ and hence of $\K(\dalfield{A})$.
\end{proof}

\section{Metastability in \texorpdfstring{$\VDF$}{VDF}}
\label{sec:metastab VDF}

In this section we prove that maximally complete models of $\VDF$ are metastability bases. The main issue is that we can only prove Proposition\,\ref{prop:prol inv st dom} when we control the $\Llt$-algebraic closure of the parameters inside the stable part. Thus we cannot apply it blindly to sets of the form $C\Valgp(\dcl[\LGD](Cc))$. However, in $\ACVF$, we have a more precise description of types over maximally complete fields:

\begin{proposition}[st dom max](\cite[Remark\,12.19]{HasHruMac-Book})
Let $M\models\ACVF$, $C\subseteq M$ be maximally complete and algebraically closed, $a\in\K(M)$ be a tuple and $H := \Valgp(\dcl[\Llt](Ca))$. Then $\tp(a)[CH]$ is stably dominated via $\ltf(C(a))$, where $\ltf(x)$ is seen as an element of $\lt[\val(x)] \subset \St[CH]$.
\end{proposition}

It follows that, to prove the existence of metastability bases, we have to study the $\LltD$-algebraic closure in $\eq{\lt}_{\Llt}$. In Proposition\,\ref{prop:acl RV}, we showed that we have control over the $\LltD$-algebraic closure hence it suffices to prove that $\lt$ with its $\Llt$-induced structure eliminates imaginaries. As a matter of fact, we only need to prove elimination for the $\Llt$-structure induced on $\lt[H] = \bigcup_{\gamma\in H} \lt[\gamma]$ where each fibre is a distinct sort.

In \cite{Hru-GpIm}, Hrushovski studies such structures. He shows, in \cite[Lemma\,5.6]{Hru-GpIm}, that they eliminate imaginaries. Note that, as every $\lt[\gamma]$ is one dimensional, these structures have flags.

\begin{proposition}[meta basis VDF]
Let $M\models\VDF$, $C\subseteq \K(M)$ be a maximally complete algebraically closed differential subfield and $a\in \K(M)$. Then, the type $\tp[\LltD](a)[C\Valgp(\dcl[\LltD](Ca))]$ is stably dominated.
\end{proposition}

\begin{proof}
Let $H := \Valgp(\dcl[\LltD](Ca))$. By Proposition\,\ref{prop:st dom max}, \[H =\DIV{\Valgp(\dalfield{Ca})} = \Valgp(\dcl[\Llt](C\prol[\omega](a))).\] By Proposition\,\ref{prop:st dom max}, $\tp[\Llt](\prol[\omega](a))[CH]$ is stably dominated via $\ltf(C(a)) \subseteq \lt[H]$. Moreover, by Proposition\,\ref{prop:acl RV}, we have
\[\left(\lt[H]<\Llt>\right)^{\eqop}(\acleq[\LltD](CH)) = \lt[H](\acl[\LltD](CH)) =  \lt[H](\acl[\Llt](CH)).\] Proposition\,\ref{prop:prol inv st dom}, now allows us to conclude that $\tp[\LltD](a)[CH]$ is stably dominated.
\end{proof}

\begin{corollary}[VDF meta]
The theory $\VDF$ admits metastability bases.
\end{corollary}

\begin{proof}
By Proposition\,\ref{prop:meta basis VDF}, we only have to show that  any $A\subseteq M\models\VDF$ is contained in a (small) maximally complete $C\subseteq\K(M)$. As the sort $\K$ is dominant, we may assume that $A\subseteq\K(A)$. Taking any lifting in $\K$ of the points in $A$, we may assume that $A\subseteq \dcl[\LGD](\K(A))$. If $\K(A)$ is not maximally complete, take $(x_{\alpha})$ to be a maximal pseudo-convergent sequence with no pseudo-limit in $\K$ and such that the order-degree of the minimal differential polynomial $P$ pseudo-solved by $(x_{\alpha})$ is minimal among all such pseudo-convergent sequences. Then the extension by any root of $P$ which is also a pseudo-limit $a$ is immediate, see \cite[Proposition\,7.32]{Sca-DValF}. Iterating this last step as many times as necessary, we obtain an immediate extension $C$ of $A$ which is maximally complete. Because $\K(C)$ is an Henselian immediate extension of an algebraically closed field, it follows that $\K(C)$ is also algebraically closed. 
\end{proof}

\part*{Definable types in enrichments of \texorpdfstring{$\ACVF$}{ACVF}}

\section{Types and uniform families of balls}
\label{sec:type ball}

Let $\LL\supseteq\Ldiv$ and $T\supseteq\ACVF$ be an $\LL$-theory that eliminates imaginaries. We assume that $T$ is $C$-minimal, i.e. every $\LL$-sort is the image of an $\LL$-definable map with domain some $\K<n>$ (we say that $\K$ is dominant) and for all $M\models T$, every $\LL(M)$-definable unary set $X\subseteq \K$ is a Boolean combination of balls. For a more extensive introduction to $C$-minimal theories, one can refer to \cite{Cub-Cmin}.

In this section, we wish to make precise the idea that, in $C$-minimal theories, $n+1$-types can be viewed as generic types of balls parametrised by realisations of an $n$-type. This is an obvious higher dimensional generalisation of the unary notion of genericity in a ball (see \cite[Definition\,2.3.4]{HasHruMac-ACVF}). To do so, we introduce a class of $\Delta$-types (see Definition\,\ref{def:gen type}) for $\Delta$ a finite set of $\LL$-formulas that will play a central role in the rest of this text. We also show that at the cost of enlarging $\Delta$, we may assume that all types are of this specific form.

The points in $\K$ are closed balls of radius $+\infty$ and $\K$ itself is an open ball of radius $-\infty$.

\begin{definition}($\Balls{l}$ and $\Ballsst{l}$)
Let $\cBallSet$ be the set of all closed balls (potentially with radius $+\infty$), $\oBallSet$ be the set of all open balls (potentially with radius $-\infty$) and $\BallSet := \cBallSet\cup\oBallSet$. For $l\in\Nn_{>0}$. We define $\Balls{l} := \{B\subseteq\BallSet\mid\card{B}\leq l\}$. We also define $\Ballsst{l} := \{B\in\Balls{l}\mid$ all the balls in $B$ have the same radius and they are either all open or all closed$\}$.
\end{definition}

The index $\st$ stands for "same radius".

\begin{notation}
For all $B\in\Balls{l}$, we will be denoting by $\Points(B)$ the set $\bigcup_{b\in B}b$, i.e. the set of valued field points in the balls of $B$. Because the balls can be nested, $\Points$ is not an injective function. However, in each fibre of $\Points$ there is a unique element with minimal cardinality, the one where there is no intersection between the balls. We denote by $\balls$ this section of $\Points$.
\end{notation}

Points in $\Ballsst{l}$ behave more or less like balls. For example if $B_{1}$, $B_{2}\in\Ballsst{l}$ are such that $\Points(B_{1})\subset \Points(B_{2})$, then either all the balls in $B_{1}$ have smaller radius than the balls in $B_{2}$ or if they have equal radiuses, then the balls in $B_{1}$ must be open and those in $B_{2}$ must be closed.

\begin{definition}(Generalised radius)
Let $B\in\Ballsst{l}\sminus\{\emptyset\}$. We define the generalised radius of $B$ (denoted $\grad(B)$) to be the pair $(\gamma,0)$ when the balls in $B$ are closed of radius $\gamma$ and the pair $(\gamma,1)$ when they are open of radius $\gamma$. The set of generalised radiuses is ordered lexicographically. We define the generalised radius of $\emptyset$ to be $(+\infty,1)$, i.e. greater than any generalised radius of non empty $B\in\Ballsst{l}$.
\end{definition}

\begin{proposition}[inter balls]
Let $(B_{i})_{i\in I} \subseteq \Ballsst{l}$. Assume that there exists $i_{0}$ such that the balls in $B_{i_{0}}$ have generalised radius greater or equal than all the other $B_{i}$. Then $\balls(\bigcap_{i}\Points(B_{i})) \subseteq B_{i_{0}}$. Moreover, there exists $(i_{j})_{0<j\leq l}\in I$ such that $\bigcap_{i}\Points(B_{i}) = \bigcap_{j=0}^{l}\Points(B_{i_{j}})$.
\end{proposition}

\begin{proof}
For any $b\in B_{i_{0}}$, if $\bigcap_{i}\Points(B_{i})\cap b\neq\emptyset$ then $b\subseteq \bigcap_{i}\Points(B_{i})$. It follows that:
\[\bigcap_{i} \Points(B_{i}) = \Points(\{b\in B_{i_{0}}\mid b\cap\bigcap_{i} B_{i}\neq\emptyset\}).\] Thus $\balls(\bigcap_{i}\Points(B_{i})) \subseteq B_{i_{0}}$. Moreover, if $\bigcap_{i}\Points(B_{i})\cap b = \emptyset$, then there exists $i_{b}$ such that $b\cap \Points(B_{i_{b}}) = \emptyset$ and $\bigcap_{i}\Points(B_{i})$ can be obtained by intersecting $B_{i_{0}}$ with the $B_{i_{b}}$ of which there are at most $l$.
\end{proof}

\begin{definition}[di]($d_{i}(B_{1},B_{2})$)
Let $b_{1}$, $b_{2}\in\BallSet$. When $b_{1}\cap b_{2} = \emptyset$, we define $d(b_{1},b_{2})$ to be $\val(x_{1}-x_{2})$, where $x_{i}\in b_{i}$, which does not depend on the choice of the $x_{i}$. When $b_{1}\cap b_{2} \neq \emptyset$, we define $d(b_{1},b_{2}) = \min\{\rad(b_{1}),\rad(b_{2})\}$, where $\rad$ denotes the radius.

For all $B_{1}$, $B_{2}\in\Balls{l}$, let us define $D(B_{1},B_{2}) := \{d(b_{1},b_{2})\mid b_{1}\in B_{1}$ and $b_{2}\in B_{2}\}$. Let us list the elements in $D(B_{1},B_{2})$ as $d_{1} > d_{2} > \cdots > d_{k}$. For all $i \leq k$, we define $d_{i}(B_{1},B_{2}) := d_{i}$.
\end{definition}

When $B_{1}$, $B_{2}\in\Ballsst{l}$, we also define $d_{0}(B_{1},B_{2}) := \min\{\rad(B_{1}),\rad(B_{2})\}$; it is equal to $d_{1}(B_{1},B_{2})$ when $\Points(B_{1})\cap\Points(B_{2})\neq\emptyset$. Later, for coding purposes, we might want $d_{i}(B_{1},B_{2})$ to be defined for all $i\leq l^{2}$ in which case, for $i >k$, we set $d_{i}(B_{1},B_{2}) = d_{k}$.\medskip

Let $M\models T$, $F = (F_{\lambda})_{\lambda\in\Lambda}$ be an $\LL(M)$-definable family of functions $\K<n>\to\Ballsst{l}$ and $\Delta(x,y;t)$ be a finite set of $\LL$-formulas where $x\in\K<n>$, $y\in\K$ and $t$ is a tuple of variables. To simplify notation, we will be denoting $\Points(F_{\lambda}(x))$ by $\points{F_{\lambda}}(x)$. We define $\Funform{\Delta}{F}(x,y;t,\lambda)$ to be the set for formulas $\Delta(x,y;t)\cup\{y\in \points{F}_{\lambda}(x)\wedge\lambda\in\Lambda\}$.

Note that if $n=0$, all of what we prove in this section and in Section\,\ref{sec:imp def} holds. It is, in fact, much more straightforward because we are considering fixed balls instead of parametrised balls.

\begin{definition}($\Delta$ adapted to $F$)
We say that $\Delta$ is adapted to $F$ if for all $p\in\TP[x,y]<\Delta>(M)$, $\lambda$, $(\mu_{i})_{0\leq i< l}\in\Lambda(M)$ and $i\leq l^{2}$, $p(x,y)$ decides:
\begin{thm@enum}
\item If $\points{F_{\lambda}}(x) \square \bigcup_{0\leq i< l}\points{F_{\mu_{i}}}(x)$ (respectively $F_{\lambda}(x) \square \bigcup_{0\leq i<l} F_{\mu_{i}}(x)$), where $\square\in\{=,\subseteq\}$;
\item If $\points{F_{\lambda}}(x) = \points{F_{\mu_{1}}}(x) \cap \points{F_{\mu_{2}}}(x)$;
\item If the balls in $F_{\lambda}(x)$ are closed;
\item If $\rad(F_{\lambda_{1}}(x)) \square d_{i}(F_{\mu_{1}}(x),F_{\mu_{2}}(x))$ where $\square\in\{=,\leq\}$.
\end{thm@enum}
Moreover, we require that there exist $\lambda_{\emptyset}$, $\lambda_{\K}\in\Lambda$ such that for all $x\in\K<n>$, $F_{\lambda_{\emptyset}}(x) = \emptyset$ and $F_{\lambda_{\K}}(x) = \{\K\}$.
\end{definition}

Note that none of the above formulas actually depend on $y$ so what is really relevant is not $p$ but the closed set induced by $p$ in $\TP[x]<\LL>(M)$. Until Proposition\,\ref{prop:exists F Delta}, let us assume that $\Delta$ is adapted to $F$. Let $p\in\TP[x,y]<\Delta>(M)$.

\begin{definition}(Generic intersection)
We say that $F$ is closed under generic intersection over $p$ if for all $\lambda_{1}$, $\lambda_{2}\in\Lambda(M)$, there exists $\mu\in\Lambda(M)$ such that 
\[p(x,y)\vdash \points{F_{\mu}}(x) = \points{F_{\lambda_{1}}}(x)\cap\points{F_{\lambda_{2}}}(x).\]
\end{definition}

Let us assume, until Proposition\,\ref{prop:exists F Delta}, that $F$ is closed under generic intersection over $p$.

\begin{definition}(Generic irreducibility)
For all $\lambda\in\Lambda(M)$, we say that $F_{\lambda}$ is generically irreducible over $p$ if for all $\mu\in\Lambda(M)$, if $p(x,y)\vdash F_{\mu}(x)\subseteq F_{\lambda}(x)$ and $p(x,y)\vdash F_{\mu}(x)\neq\emptyset$ then $p(x,y)\vdash F_{\mu}(x) = F_{\lambda}(x)$.

We say that $F$ is generically irreducible over $p$ if for every $\lambda\in\Lambda(M)$, $F_{\lambda}$ is generically irreducible over $p$.
\end{definition}

Let us now show that generically irreducible families of balls behave nicely under generic intersection.

\begin{proposition}[inter gen irr]
Let $\lambda_{1}$, $\lambda_{2}\in\Lambda(M)$ be such that $F_{\lambda_{1}}$ and $F_{\lambda_{2}}$ are generically irreducible over $p$ and $p(x,y)$ implies that the balls in $F_{\lambda_{1}}(x)$ have smaller or equal generalised radius than the balls in $F_{\lambda_{2}}(x)$. Then either $p(x,y)\vdash \points{F_{\lambda_{1}}}(x)\cap \points{F_{\lambda_{2}}}(x)= \emptyset$ or $p(x,y)\vdash \points{F_{\lambda_{1}}}(x)\cap \points{F_{\lambda_{2}}}(x) = \points{F_{\lambda_{1}}}(x)$.
\end{proposition}

\begin{proof}
Let $(a,c)\models p$. By Proposition\,\ref{prop:inter balls}, we have that $\balls(\points{F_{\lambda_{1}}}(a)\cap \points{F_{\lambda_{2}}}(a))  \subseteq \points{F_{\lambda_{1}}}(a)$. By generic intersection, there exists $\mu$ such that $p(x,y)\vdash \points{F_{\mu}}(x) = \points{F_{\lambda_{1}}}(x)\cap \points{F_{\lambda_{2}}}(x)$. Then $F_{\mu}(a)\subseteq F_{\lambda_{1}}(a)$. Hence, if $F_{\mu}(a)\neq\emptyset$, $F_{\mu}(a) = F_{\lambda_{1}}(a)$.
\end{proof}

\begin{corollary}[inter gen irr cor]
Assume $p$ is $\LL(M)$-definable. Then $\Lambda_{p} := \{\lambda\in\Lambda\mid F_{\lambda}$ is generically irreducible over $p\}$ is $\LL(M)$-definable and the $\LL(M)$-definable family $(F_{\lambda})_{\lambda\in\Lambda_{p}}$ is closed under generic intersection over $p$.
\end{corollary}

\begin{proof}
The definability of $\Lambda_{p}$ is a consequence of the definability of $p$. The closure of $(F_{\lambda})_{\lambda\in\Lambda_{p}}$ under generic intersection follows from Proposition\,\ref{prop:inter gen irr}.
\end{proof}

Until Proposition\,\ref{prop:exists F Delta}, let us also assume that $F$ is generically irreducible over $p$. 

\begin{definition}[gen type](Generic type of $E$ over $p$)
Let $E\subset\Lambda(M)$. We define $\Gen{E}[p](x,y)$, the $(\Delta,F)$-generic type of $E$ over $p$, to be the following $\Funform{\Delta}{F}$-type over $M$:
\[\begin{eqn}
p(x,y)&\cup&\{y\in \points{F_{\lambda}}(x)\mid\lambda\in E\}\\
&\cup&\{y\nin \points{F_{\mu}}(x)\mid\mu\in\Lambda(M)\text{ and for all }\lambda\in E,\,p(x,y)\vdash \points{F_{\mu}}(x)\subset \points{F_{\lambda}}(x)\}.
\end{eqn}\]
\end{definition}

Note that, most of the time, $\Delta$ and $F$ will be obvious from the context, so it will not be an issue that the notation $\Gen{E}[p]$ mentions neither $\Delta$ nor $F$.

\begin{proposition}[Gen compl]
Let $E\subset\Lambda(M)$ be such that $\Gen{E}[p]$ is consistent, then $\Gen{E}[p]$ generates a complete $\Funform{\Delta}{F}$-type over $M$.
\end{proposition}

\begin{proof}
Pick any $\mu\in\Lambda(M)$. If there is $\lambda\in E$ such that $p(x,y)\vdash \points{F_{\mu}}(x)\cap \points{F_{\lambda}}(x) = \emptyset$, then $\Gen{E}[p](x,y)\vdash y\nin \points{F_{\mu}}(x)$. If there exists $\lambda\in E$ such that $p(x,y)\vdash\points{F_{\lambda}}(x)\subseteq\points{F_{\mu}}(x)$, then $\Gen{E}[p](x,y)\vdash y\in \points{F_{\mu}}(x)$. If non of these cases apply, for all $\lambda\in E$, $p(x,y)\vdash\points{F_{\mu}}(x)\subset\points{F_{\lambda}}(x)$ and $\Gen{E}[p](x,y)\vdash y\nin \points{F_{\mu}}(x)$.
\end{proof}

When it is consistent, we will identify $\Gen{E}[p]$ with the type it generates.

\begin{remark}
Any $q\in\TP[x,y]<\Funform{\Delta}{F}>(M)$ is of the form $\Gen{E}[p]$. Indeed, let $p := \restr{q}{\Delta}$ and $E = \{\lambda\in\Lambda(M)\mid q(x,y)\vdash y\in F_{\lambda}(x)\}$. Then, quite clearly, $q = \Gen{E}[p]$.
\end{remark}

Let us show that any finite set of formulas with variables in $\K^{n+1}$ can be decided by some $\Funform{\Delta}{F}$ for well chosen $\Delta$ and $F$.

\begin{proposition}[exists F Delta](Reduction to $\Funform{\Delta}{F}$)
For all finite sets $\Theta(x,y;t)$ of $\LL$-formulas, where $x\in\K<n>$ and $y\in\K$, there exists an $\LL$-definable family $(F_{\lambda})_{\lambda\in\Lambda}$ of functions $\K<n>\to\Balls{l}$ and a finite set of $\LL$-formulas $\Delta(x;s)$ such that any $\Funform{\Delta}{F}$-type decides all the formulas in $\Theta$.
\end{proposition}

\begin{proof}
Let $\phi(x,y;t)$ be a formula in $\Theta$. As $T$ is $C$-minimal, for all tuples $a\in\K$ and $c\in M$, the set $\phi(a,M;c)$ has a canonical representation as Swiss cheeses, i.e. it is of the form $\bigcup_{i} (b_{i}\sminus b_{i,j})$ where the $b_{i}$ and $b_{i,j}$ are algebraic over $a c$. In particular, there exists $l\in\Nn_{>0}$  and $\LL(c)$-definable functions $H_{\phi,c} : \K<n> \to \Balls{l}$ and $G_{\phi,c} : \K<n> \to \Balls{l}$ such that $M\models \forall y\,(y\in\points{H_{\phi,c}}(a)\sminus \points{G_{\phi,c}}(a) \iffform \phi(a,y;c))$. By compactness, we can find finitely many $\LL$-definable families $(H_{i,\phi,c})_{c\in M}$ and $(G_{i,\phi,c})_{c\in M}$ of functions $\K<n>\to\Balls{l_{i,\phi}}$ such that for any choice of $c$ and $a$ there is an $i$ such that $\phi(a,y;c) \iffform y\in\points{H_{i,\phi,c}}(a)\sminus \points{G_{i,\phi,c}}(a)$. Choosing $l$ to be the maximum of the $l_{i,\phi}$ and using some coding trick, one can find an $\LL$-definable family $(F_{\lambda})_{\lambda\in\Lambda}$ of functions $\K<n>\to\Balls{l}$ such that for any $\phi\in\Theta$, $i$ and $c$ we find $\mu$, $\nu\in\Lambda$ such that $H_{i,\phi,c} = F_{\mu}$ and $G_{i,\phi,c} = F_{\nu}$.

Now let $\Delta(x;t,\mu,\nu) = \{\forall y\,(\phi(x,y;t) \iffform y\in \points{F_{\mu}}(x)\sminus \points{F_{\nu}}(x)) \mid \phi\in\Theta\}$. Then for any $p\in\TP[x,y]<\Funform{\Delta}{F}>(M)$, $\phi\in\Theta$ and tuple $c\in M$, there exists $\mu$, $\nu\in\Lambda(M)$ such that $p(x,y)\vdash \phi(x,y;c) \iffform y\in \points{F_{\mu}}(x)\sminus \points{F_{\nu}}(x)$ and either $p(x,y)\vdash y\in \points{F_{\mu}}(x)\wedge y\nin \points{F_{\nu}}(x)$ in which case $p(x,y)\vdash \phi(x,y;c)$ or not, in which case $p(x,y)\vdash \neg\phi(x,y;c)$.
\end{proof}

Now, let us show that we can refine any $\Delta$ and $F$ into a family verifying all previous hypotheses.

\begin{proposition}[F st balls](Reduction to $\Ballsst{l}$)
Let $A\subseteq M$ and $(F_{\lambda})_{\lambda\in\Lambda}$ be an $\LL(A)$-definable family of functions $\K<n>\to\Balls{l}$. Then there exists an $\LL(A)$-definable family $(G_{\omega})_{\omega\in\Omega}$ of functions $\K<n>\to\Ballsst{l}$ such that for all $\lambda$, there exist $(\omega_{i})_{0\leq i < l}$ such that $F_{\lambda}(x) = \bigcup_{i}G_{\omega_{i}}(x)$ and for all $\omega$ there exists $\lambda$ such that $G_{\omega}(x)\subseteq F_{\lambda}(x)$.
\end{proposition}

\begin{proof}
For all $\lambda\in\Lambda$, $0 < i \leq l$ and $j=0,1$,  we define $G_{\lambda,i,j}(x) := \{b\in F_{\lambda}(x)\mid b$ is open if $j=0$, closed otherwise and $b$ has the $i$-th smallest radius among the balls in $F_{\lambda}(x)\}$. As $i$ and $j$ only take finitely many values, $G = (G_{\omega})_{\omega\in\Omega}$ can indeed be viewed as an $\LL(A)$-definable family. Then for all $x$, $G_{\omega}(x)\in\Ballsst{l}$. For all $x$ and $\lambda$, $G_{\lambda,i,j}(x)\subseteq F_{\lambda}(x)$ and $F_{\lambda}(x) = \bigcup_{i,j}G_{\lambda,i,j}(x)$. Moreover, at most $l$ of them are non empty.
\end{proof}

\begin{definition}(Generic complement)
We say that $F$ is closed under generic complement over $p$ if for all $\lambda$, $\mu\in\Lambda(M)$ such that $p(x)\vdash F_{\mu}(x)\subseteq F_{\lambda}(x)$, there exists $\kappa\in\Lambda(M)$ such that \[p(x)\vdash F_{\lambda}(x) = F_{\mu}(x)\dunion F_{\kappa}(x).\]
\end{definition}

Note that $p$ can decide any such statement as it is equivalent to $F_{\lambda}(x) = F_{\mu}(x) \cup F_{\kappa}(x)$ and $\points{F_{\mu}}(x)\cap \points{F_{\kappa}}(x) = \emptyset$.

\begin{lemma}[cover irred]
Let $F = (F_{\lambda})_{\lambda\in\Lambda}$ be an $\LL(M)$-definable family of functions $\K<n>\to\Ballsst{l}$, $\Delta(x;t)$ a finite set of $\LL$-formulas adapted to $F$ and $p\in\TP[x]<\Delta>(M)$. Assume that $F$ is closed under generic complement over $p$. Let $\Lambda_{p} := \{\lambda\in\Lambda\mid F_{\lambda}$ is generically irreducible over $p\}$, then for all $\lambda\in\Lambda(M)$ there exists $(\lambda_{i})_{0\leq i < l}\in\Lambda_{p}(M)$ such that $p(x)\vdash F_{\lambda}(x) = \bigcup_{i} F_{\lambda_{i}}(x)$.
\end{lemma}

\begin{proof}
Let $x\models p$. We work by induction on $\card{F_{\lambda}(x)}$. If there exists $\mu\in\Lambda(M)$ such that $F_{\mu}(x)\subset F_{\lambda}(x)$ and $F_{\mu}(x)\neq \emptyset$, then there exists $\kappa\in\Lambda(M)$ such that $F_{\lambda}(x) = F_{\mu}(x)\dunion F_{\kappa}(x)$. We now apply the induction hypothesis to $F_{\mu}(x)$ and $F_{\kappa}(x)$. Finally, because $\card{F_{\lambda}(x)}\leq l$, we cannot cut it in more than $l$ distinct pieces.
\end{proof}

\begin{proposition}[exists F Delta nice](Reduction to irreducible families)
Let $A\subseteq M$, $(F_{\lambda})_{\lambda\in\Lambda}$ be an $\LL(A)$-definable family of functions $\K<n>\to\Ballsst{l}$ and $\Delta(x;t)$ a finite set of $\LL$-formulas. Then, there exists an $\LL(A)$-definable family $(G_{\omega})_{\omega\in\Omega}$ of functions $\K<n>\to\Ballsst{l}$ and a finite set of $\LL$-formulas $\Theta(x;t,s)\supseteq\Delta(x;t)$ such that $\Theta$ is adapted to $G$ and for any $p\in\TP[x]<\Theta>(M)$:
\begin{thm@enum}
\item\label{nice:inter} $G$ is closed under generic intersection and complement over $p$;
\item\label{nice:sub} For all $\omega\in\Omega(M)$ there is $\lambda\in\Lambda(M)$ such that $p(x)\vdash G_{\omega}(x)\subseteq F_{\lambda}(x)$;
\item\label{nice:equal} For all $\lambda\in\Lambda(M)$, there is $\omega\in\Omega(M)$ such that $p(x)\vdash F_{\lambda}(x) = G_{\omega}(x)$;
\item\label{nice:irr cover} For all $\omega\in\Omega(M)$, there is $(\omega_{i})_{0\leq i <l}\in\Omega_{p}(M)$ such that $p(x)\vdash G_{\omega}(x) = \bigcup_{i}G_{\omega_{i}}(x)$;
\end{thm@enum}
where $\Omega_{p} := \{\omega\in\Omega\mid G_{\omega}$ is generically irreducible over $p\}$.
\end{proposition}

\begin{proof}
Adding them if necessary, we may assume that $F$ contains the constant functions equal to $\emptyset$ and $\{\K\}$ respectively. For all $\uple{\lambda}\in\Lambda^{l+1}$, let $H_{\overline{\lambda}}(x) := \balls(\bigcap_{0\leq i\leq l} \points{F_{\lambda_{i}}}(x))$.  It follows from Proposition\,\ref{prop:inter balls}, that $H = (H_{\uple{\lambda}})_{\uple{\lambda}\in\Lambda^{l+1}}$ is well-defined and that \ref{nice:sub} holds for $H$. Adding finitely many formulas to $\Delta(x;t)$, we obtain $\Xi(x;s)$ which is adapted to $H$. Let $p\in\TP[x]<\Xi>(M)$. Proposition\,\ref{prop:inter balls} also implies that for a given $x$, the intersection of any number of $\points{F_{\lambda}}(x)$ is given by the intersection of $l+1$ of them. Hence it is an instance of $H$. As $\Xi$ is adapted to $H$, we have proved that $H$ is closed under generic intersection over any $\Xi$-type $p$. Condition\,\ref{nice:equal} also clearly holds for $H$.

Let $B\in\Ballsst{l}$, we define $B^{1}$ to be $B$ and $B^{0}$ to be its complement (in $\BallSet$). As previously, to simplify notation, for $\epsilon\in\{0,1\}$, we will write $H^{\epsilon}_{\mu}(x)$ for $(H_{\mu}(x))^{\epsilon}$.

\begin{claim}[bool comb balls]
Let $B\in\Ballsst{l}$. Any Boolean combination of sets $(C_{i})_{i\leq r}\subseteq B$ (where we take the complement in $B$, i.e. $C^{0}\cap B$) lives in $\Ballsst{l}$ and can be written as $\bigcap_{j<l}\bigcup_{k<l}(C_{j,k}^{\epsilon_{j,k}}\cap B)$ where the $C_{j,k}$ are taken among the $C_{i}$ and $\epsilon_{j,k}\in\{0,1\}$.
\end{claim}

\begin{proof}
Such a Boolean combination lives in $\Ballsst{l}$ because it is a subset of $B$. The fact that it can be written as $\bigcap_{j}\bigcup_{k}(C_{j,k}^{\epsilon_{j,k}}\cap B)$ is just the existence of the conjunctive normal form. Moreover, as in Proposition\,\ref{prop:inter balls}, any intersection $\bigcap_{k}C_{j,k}^{\epsilon_{j,k}}\cap B$ for fixed $j$ can be rewritten as the intersection of at most $l$ of them (for each ball from $B$ missing from the intersection, choose a $k$ such that this ball is not in $C_{j,k}^{\epsilon_{j,k}}\cap B$). Similarly, the union can be rewritten as the union of at most $l$ of them by choosing, for every $b\in B$ which appears in the union a $j$ such that $b$ appears in $\bigcup_{k}(C_{j,k}^{\epsilon_{j,k}}\cap B)$.
\end{proof}

For all $\nu\in\Lambda^{l+1}$, $\uple{\mu}\in(\Lambda^{l=1})^{l^{2}}$ and $\uple{\epsilon}\in 2^{l^{2}}$, let $G_{\nu,\overline{\mu},\overline{\epsilon}}(x) = \bigcap_{i<l}\bigcup_{j<l}((H^{\epsilon_{i,j}}_{\mu_{i,j}}(x))\cap H_{\nu}(x))$ whenever all the $H_{\mu_{i,j}}\subseteq H_{\nu}(x)$. Otherwise, let $G_{\nu,\overline{\mu},\overline{\epsilon}}(x) = H_{\nu}(x)$. Adding some more formulas to $\Xi$, we obtain a finite set of formulas $\Theta(x;t,s,u)$ which is adapted to $G$. It is clear that \ref{nice:sub} and \ref{nice:equal} still hold. Furthermore,
\[\points{G_{\nu,\overline{\mu},\overline{\epsilon}}}(x)\cap \points{G_{\sigma,\overline{\tau},\overline{\eta}}}(x) = \bigcap_{i,k}\bigcup_{j,r}(\Points(H_{\mu_{i,j}}^{\epsilon_{i,j}}(x))\cap \Points(H_{\tau_{k,r}}^{\eta_{k,r}}(x))\cap \points{H_{\nu}}(x)\cap \points{H_{\sigma}}(x)).\]
As $H$ is closed under generic intersection there exists $\rho$ such that $\points{H_{\rho}}(x) = \points{H_{\nu}}(x)\cap \points{H_{\sigma}}(x)$. By Proposition\,\ref{prop:inter balls}, $\balls(\Points(H_{\mu_{i,j}}^{\epsilon_{i,j}}(x))\cap \points{H_{\rho}}(x))\subseteq H_{\rho}(x)$ and $\balls(\Points(H_{\tau_{k,r}}^{\eta_{k,r}}(x))\cap \points{H_{\rho}}(x))\subseteq H_{\rho}(x)$. We can conclude by Claim\,\ref{claim:bool comb balls} that $G$ is also closed under generic intersection over $p$. Similarly we show that whenever $G_{\nu,\overline{\mu},\overline{\epsilon}}(x)\subseteq G_{\sigma,\overline{\tau},\overline{\eta}}(x)$ then $G^{0}_{\nu,\overline{\mu},\overline{\epsilon}}(x)\cap G_{\sigma,\overline{\tau},\overline{\eta}}(x)$ is also an instance of $G$, i.e. $G$ is closed under generic complement over $p$. Hence \ref{nice:irr cover} is proved in Lemma\,\ref{lem:cover irred}.
\end{proof}

\section{Quantifiable types} 
\label{sec:imp def}

Let us begin with the example that motivates the definition of quantifiable types. Let $b$ be an open ball in some model of $\ACVF$ and $\Gen{b}$ be its generic type. Let $X$ be any set definable in an enrichment of $\ACVF$. Then all the realisations of $\Gen{b}$ are in $X$, i.e. $\Gen{b}\vdash x\in X$, if and only if there exists $b'\in\BallSet$ such that $b'\subset b$ and $b\sminus b' \subseteq X$. Thus, although for most definable sets $X$, both $X$ and its complement are consistent with $\Gen{b}$, if it happens that any realisation of $\Gen{b}$ is in $X$, then there is a formula which says so. We have just shown that $\Gen{b}$ is quantifiable as a partial $\tL$-type (see Definition\,\ref{def:imp def}) for any enrichment $\tL$ of $\ACVF$. If $(b_{i})_{i\in I}$ is a strict chain of balls, i.e. $P := \bigcap_{i} b_{i}$ is not a ball, the exact same proof shows that the generic type of $P$ is also quantifiable as a partial $\tL$-type, if $P$ is $\tL$-definable.

If $b$ is a closed ball, the situation is somewhat more complicated because $\Gen{b}(x) \vdash x\in X$ if and only if there exists finitely many maximal open subballs $(b_{i})_{0\leq i < k}$ of $b$ such that for all $x\in\K$, $x\in b\sminus \bigcup_{i} b_{i}$ implies $x\in X$. Because the set of maximal open subballs of a given ball is internal to the residue field, to obtain that $\Gen{b}$ is quantifiable (as a partial $\tL$-type), we need to know that the $\tL$-induced structure on $\res$ eliminates $\exists^{\infty}$ to bound the number of maximal open subballs we have to remove. Recall that an $\LL$-theory $T$ eliminates $\exists^{\infty}$ if for every $\LL$-formula $\phi(x;s)$ there is an $n\in\Nn$ such that for all $M\models T$ and $m\in M$, if $\card{\phi(M;m)} < \infty$ then $\card{\phi(M;m)} \leq n$.

The notion of quantifiable type will play a fundamental role in Section\,\ref{sec:approx}. The main result of this section is Corollary\,\ref{cor:imp def gen} which says that, under some more hypothesis on the families of parametrised balls, the types of the form $\Gen{E}[p]$ (see Definition\,\ref{def:gen type}) are quantifiable if $E$ is definable and $p$ is quantifiable. The proof is essentially a parametrised version of the argument above. We then prove that we can refine families of parametrised balls so that they have the necessary properties.

Let $\LL$ be a language and $M$ an $\LL$-structure.

\begin{definition}[imp def](Quantifiable partial $\LL$-types)
Let $p$ be a partial $\LL(M)$-type. We say that $p$ is quantifiable if for all $\LL$-formulas $\phi(x;s)$ there exists an $\LL(M)$-formula $\theta(s)$ such that for all tuples $m\in M$, \[M\models \theta(m)\text{ if and only if }p(x)\vdash\phi(x;m).\]
Let $A\subseteq M$. If we want to specify that $\theta$ is an $\LL(A)$-formula, we will say that $p$ is $\LL(A)$-quantifiable.
\end{definition}

\begin{remark}
\begin{thm@enum}
\item A type $p(x)$ is quantifiable if we can quantify universally and existentially over realisations of $p$, that is for every $\LL$-formula $\phi(x;y)$, "for all $x\models \tprestr{p}{y}$, $\phi(x;y)$ holds" and "there exists an $x\models \tprestr{p}{y}$ such that $\phi(x;y)$ holds" are both first order formulas. Hence the name.
\item There are various ways in which to extend definability to partial types depending on two things: do we want the defining scheme to be ind-definable, pro-definable or definable? And do we want the closure under implication of the partial type also to be definable? Quantifiable partial types correspond to the case where the closure under implication of the type has a definable defining scheme. Although these different notions have often been indistinctively called definability, we feel that it is better to try to distinguish them.
\item\label{rem:imp def is def} The partial types we will consider here are $\Delta$-types for some set $\Delta(x;t)$ of $\LL$-formulas. Note that if $p\in\TP[x]<\Delta>(M)$ is $\LL(A)$-quantifiable, it is $\LL(A)$-definable as a $\Delta$-type, i.e. for any formula $\phi(x;t)\in\Delta$, there is an $\LL(A)$-formula $\defsc{p}{x}\phi(x;t) = \theta(t)$ such that for all tuples $m\in M$, $\phi(x;m)\in p$ if and only if $M\models\defsc{p}{x}\phi(x;m)$. In particular, $p$ has a canonical extension $\tprestr{p}{N}$ to any $N\supsel M$ defined using the same defining scheme. If $M$ was sufficiently saturated, this canonical extension is also  $\LL(A)$-quantifiable.
\end{thm@enum}
\end{remark}

As previously, let now $\LL\supseteq\Ldiv$, $T\supseteq\ACVF$ be a $C$-minimal $\LL$-theory which eliminates imaginaries, $\Real$ be the set of $\LL$-sorts, $\tL$ be an enrichment of $\LL$, $\tT$ an $\tL$-theory containing $T$, $\tM\models\tT$ and $M := \Langrestr{\tM}{\LL}$. We will also be assuming that $\res$ is stably embedded in $\tT$ and that the induced theory on $\res$ eliminates $\exists^{\infty}$. Until the end of the section, quantifiability of types will refer to quantifiability as partial $\tL$-types.

Let $\tA\subseteq\tM$, $A := \Real(\tA)$, $F = (F_{\lambda})_{\lambda\in\Lambda}$ be an $\LL(A)$-definable family of functions $\K<n>\to\Ballsst{l}$ and $\Delta(x,y;t)$ a finite set of $\LL$-formulas where $x\in\K<n>$ and $y\in\K$. Let $p\in\TP[x,y]<\Delta>(M)$ be definable. Assume that $\Delta$ is adapted to $F$ and that $F$ is is generically irreducible and closed under generic intersection over $p$.

\begin{definition}[gen cov](Generic covering property)
We say that $F$ has the generic covering property over $p$ if for any $E\subseteq\Lambda(M)$ and any finite set $(\lambda_{i})_{0\leq i < k}\in\Lambda(M)$ such that for all $\mu\in E$, $p(x,y)\vdash \points{F_{\lambda_{i}}}(x) \subset \points{F_{\mu}}(x)$, there exists $(\kappa_{j})_{0\leq j<l}\in\Lambda(M)$ such that:
\begin{thm@enum}
\item For all $j$, $p(x,y)\vdash$ "the balls in $F_{\kappa_{j}}(x)$ are closed";
\item For all $\mu\in E$ and $j$, $p(x,y)\vdash \points{F_{\kappa_{j}}}(x) \subseteq \points{F_{\mu}}(x)$;
\item For all $i$, $p(x,y)\vdash \points{F_{\lambda_{i}}}(x) \subseteq \bigcup_{j}\points{F_{\kappa_{j}}}(x)$;
\end{thm@enum}
\end{definition}

Note that if $E = \{\lambda_{0}\}$ and $p(x,y)\vdash$ "the balls in $F_{\lambda_{0}}(x)$ are closed", then the generic covering property holds trivially as it suffices to take all $\kappa_{j} = \lambda_{0}$. It will only be interesting if $p(x,y)\vdash$ "the balls in $F_{\lambda_{0}}(x)$ are open" or $E$ does not have a smallest element over $p$, i.e. for all $\lambda\in E$ there exists $\mu\in E$ such that $p(x,y)\vdash \points{F_{\mu}}(x) \subset \points{F_{\lambda}}(x)$.

Let $\cE\subseteq\Lambda$ be $\tL(\tA)$-definable. 

\begin{proposition}[imp def strict open]
Assume that one of the following holds:
\begin{thm@enum}
\item $\cE(\tM)$ does not have a smallest element over $p$;
\item there is a $\lambda_{0}\in \cE(\tM)$ such that for all $\lambda\in \cE(\tM)$ , $p(x,y)\vdash \points{F_{\lambda_{0}}}(x)\subseteq \points{F_{\lambda}}(x)$ and $p(x,y)\vdash$ "the balls in $F_{\lambda_{0}}(x)$ are open".
\end{thm@enum}
Assume also that $p$ is $\tL(\tA)$-quantifiable and $F$ has the generic covering property over $p$, then $\Gen{\cE(\tM)}[p]$ is $\tL(\tA)$-quantifiable.
\end{proposition}

\begin{proof}
Let $\phi(x,y;t)$ be an $\tL$-formula. If $\Gen{\cE(\tM)}[p](x,y)\vdash\phi(x,y;m)$, for some tuple $m\in \tM$, then there exists $\lambda_{0}\in \cE(\tM)$ and a finite number of $(\lambda_{i})_{0<i<k}\in\Lambda(M)$ such that for all $\mu\in \cE(\tM)$ and $i>0$,  $p(x,y)\vdash y\in\points{F_{\lambda_{0}}}(x)\sminus\bigcup_{i>0}\points{F_{\lambda_{i}}}(x) \impform \phi(x,y;m)$ and $p(x,y)\vdash \points{F_{\lambda_{i}}}(x)\subset \points{F_{\mu}}(x)$. By the generic covering property, we can find $(\kappa_{j})_{0\leq j<l}\in\Lambda(M)$ such that, for all $j$, $p(x,y)\vdash$ "the balls in $F_{\kappa_{j}}(x)$ are closed", for all $\mu\in \cE(\tM)$ and $j$, $p(x,y)\vdash \points{F_{\kappa_{j}}}(x)\subseteq \points{F_{\mu}}(x)$ and for all $i>0$, $p(x,y)\vdash\points{F_{\lambda_{i}}}(x) \subseteq \bigcup_{j}\points{F_{\kappa_{j}}}(x)$. 

If $\cE(\tM)$ does not have a smallest element over $p$, for all $\mu\in \cE(\tM)$ and $j$, we have that $p(x,y)\vdash \points{F_{\kappa_{j}}}(x)\subset \points{F_{\mu}}(x)$. If $\cE(\tM)$ has a smallest element, because the balls in $F_{\lambda_{0}}(x)$ are open and those in $F_{\kappa_{i}}(x)$ are closed, we also have $p(x,y)\vdash \points{F_{\kappa_{j}}}(x)\subset \points{F_{\lambda_{0}}}(x)$. As the $\bigcup_{j}\points{F_{\kappa_{j}}}(x)$ covers $\bigcup_{i}\points{F_{\lambda_{i}}}(x)$, it follows that:
\[p(x,y)\vdash y\in \points{F_{\lambda_{0}}}(x)\sminus\bigcup_{0\leq j<l}\points{F_{\kappa_{j}}}(x)\impform\phi(x,y;m).\]

 We have just shown that, for all tuples $m\in\tM$, $\Gen{\cE(\tM)}[p](x,y)\vdash\phi(x,y;m)$ implies that:
\[\tM\models\exists\lambda_{0}\in\cE\,\exists\uple{\kappa}\in\Lambda\,\bigwedge_{j<l}\forall\mu\in\cE\,\delta_{1}(\kappa_{j},\mu)\wedge\delta_{2}(\lambda_{0},\uple{\kappa},m)\]
where $\delta_1(\kappa,\mu)$ is an $\tL(\tA)$-formula equivalent to $p(x,y)\vdash \points{F_{\kappa}}(x)\subset \points{F_{\mu}}(x)$ and $\delta_{2}(\lambda_{0},\uple{\kappa},m)$ is an $\tL(\tA)$-formula equivalent to $p(x,y)\vdash y\in \points{F_{\lambda_{0}}}(x)\sminus\bigcup_{j<l}\points{F_{\kappa_{j}}}(x)\impform\phi(x,y;m)$.

The converse is trivial.
\end{proof}

\begin{definition}(Maximal open subball property)
We say that $F$ has the maximal open subball property over $p$ if for all $\lambda_{1}$, $\lambda_{2}\in\Lambda(M)$ such that $p(x,y)\vdash \points{F_{\lambda_{1}}}(x)\subset \points{F_{\lambda_{2}}}(x)$, there exists $(\mu_{i})_{0\leq i <l}\in\Lambda(M)$ such that:
\begin{thm@enum}
\item For all $i$, $p(x,y)\vdash$ "the balls in $F_{\mu_{i}}(x)$ are open";
\item For all $i$, $p(x,y)\vdash\rad(F_{\lambda_{2}}(x)) = \rad(F_{\mu_{i}}(x))$.
\item $p(x,y)\vdash \points{F_{\lambda_{1}}}(x) \subseteq \bigcup_{i}\points{F_{\mu_{i}}}(x)$;
\end{thm@enum}
\end{definition}

Note that when the balls in $F_{\lambda_{2}}(x)$ are open, it suffices to take all $\mu_{i} = \lambda_{2}$. Hence this property is only useful when the balls in $F_{\lambda_{2}}(x)$ are closed.

\begin{proposition}[imp def closed]
Assume that there is a $\lambda_{0}\in \cE(\tM)$ such that for all $\lambda\in \cE(\tM)$, $p(x,y)\vdash \points{F_{\lambda_{0}}}(x)\subseteq \points{F_{\lambda}}(x)$ and that $p(x,y)\vdash$ "the balls in $F_{\lambda_{0}}(x)$ are closed".  Assume also that $p$ is $\tL(\tA)$-quantifiable and that $F$ has the maximal open subball property over $p$, then the type $\Gen{\cE(\tM)}[p]$ is $\tL(\tA)$-quantifiable.
\end{proposition}

\begin{proof}
If the balls in $F_{\lambda_{0}}(x)$ have radius $+\infty$, they are singletons. By irreducibility, $F_{\lambda_{0}}(x)$ does not have any strict subset of the form $F_{\lambda}(x)$. Moreover, $\Gen{\cE(\tM)}[p]\vdash \phi(x,y;m)$ if and only if $p(x,y)\vdash y\in\points{F_{\lambda_{0}}}(x)\impform\phi(x,y;m)$. So we can conclude immediately by $\tL(\tA)$-quantifiable of $p$. We may now assume that the balls in $F_{\lambda_{0}}(x)$ have a radius different from $+\infty$. Let us begin with some preliminary results. Let $\cball{\gamma}{a}$ denote the closed ball of radius $\gamma$ around $a$.

\begin{claim}[exists infty res]
Let $(Y_{\omega,x})_{\omega\in\Omega,x\in\K<n>}$ be a definable family of sets such that for all $\omega\in\Omega$ and $x\in\K<n>$, $Y_{\omega,x}\subseteq \{b\mid b$ is a maximal open subball of some $b'\in F_{\lambda_{0}}(x)\}$. Then there exists $k\in\Nn$ such that for all $\omega\in\Omega$ and $x\in\K<n>$, either $\card{Y_{\omega,x}}\geq\infty$ or $\card{Y_{\omega,x}}\leq k$.
\end{claim}

\begin{proof}
Let $Y_{1,\omega,x,a,c} := \{b\in\BallSet\mid b\in Y_{\omega,x}$, $b$ is a maximal open subball of $\cball{\val(c)}{a}\}$. For any maximal open subball $b$ of $\cball{\val(c)}{a}$, the set $\{(x-a)/c\mid x\in b\}$ is an element of $\res$ which we denote $\resf_{a,c}(b)$. The function $\resf_{a,c}$ is one to one. Let $Y_{2,\omega,x,a,c} := \resf_{a,c}(Y_{1,\omega,x,a,c})$.

Then $Y_{2} = (Y_{2,\omega,x,a,c})_{\omega,x,a,c}$ is an $\tL(\tM)$-definable family of subsets of $\res$. By stable embeddedness of $\res$ in $T$ (as well as compactness and some coding) there exists an $\tL(\res(\tM))$-definable family $(X_{d})_{d\in D}$ where $D\subseteq\res<r>$ for some $r$ such that for all $(\omega,x,a,c)$, there exists $d\in D$ such that $Y_{2,\omega,x,a,c} = X_{d}$. As the theory induced on $\res$ eliminates $\exists^{\infty}$, there exists $s\in\Nn$ such that for all $d\in D$, either $\card{X_{d}}\geq\infty$ or $\card{X_{d}}\leq s$. It follows that for all $(\omega,x,a,c)$, either $\card{Y_{1,\omega,x,a,c}}\geq\infty$ or $\card{Y_{1,\omega,x,a,c}}\leq s$. But there are at most $l$ balls in $F_{\lambda_{0}}(x)$ and each of these balls contains infinitely or at most $s$ maximal open subballs from $Y_{\omega,x}$. Therefore, we have that for all $x$ and $\omega$, $\card{Y_{\omega,x}}\geq\infty$ or $\card{Y_{\omega,x}}\leq l s$.
\end{proof}

Let $X_{m} := \{\lambda\in\Lambda\mid p(x,y)\nvdash y\in \points{F_{\lambda}}(x)\impform\phi(x,y;m)$ and $p(x,y)\vdash$ "the balls in $F_{\lambda}(x)$ are maximal open subballs of the balls in $F_{\lambda_{0}}(x)$" $\}$. By quantifiability of $p$, $X_{m}$ is an $\tL(\tM)$-definable family. Let $Y_{m,x} := \{b\mid\exists\lambda\in X_{m},\,b\in F_{\lambda}(x)\}$. Then by Claim\,\ref{claim:exists infty res}, there exists $k$ such that for all $m$ and $x$, $\card{Y_{m,x}}<\infty$ implies $\card{Y_{m,x}}\leq k$.

Assume that $\Gen{\cE(\tM)}[p](x,y)\vdash\phi(x,y;m)$. Then, there exists $(\mu_{i})_{0\leq i <r}\in\Lambda(M)$ such that $p(x,y)\vdash \points{F_{\mu_{i}}}(x)\subset\points{F_{\lambda_{0}}}(x)$ and $p(x,y)\vdash y\in\points{F_{\lambda_{0}}}(x)\sminus\bigcup_{i}\points{F_{\mu_{i}}}(x) \impform \phi(x,y;m)$. As $F$ has the maximal open subball property over $p$ and is closed under generic intersection, we may assume that $p(x,y)\vdash$ "the balls in the $F_{\mu_{i}}(x)$ are maximal open subballs of the balls in $F_{\lambda_{0}}(x)$".

\begin{claim}
$X_{m}(M) \subseteq \{\lambda\in\Lambda(M)\mid$ for some $i$, $p(x,y)\vdash F_{\lambda}(x) = F_{\mu_{i}}(x)\}$. In particular, $\card{Y_{m,x}}<\infty$ and hence $\card{Y_{m,x}}\leq k$.
\end{claim}

\begin{proof}
Let $\lambda\in X_{m}$. There exists $x,y\models p$ such that $y\in\points{F_{\lambda}}(x)$, the balls in $F_{\lambda}(x)$ are maximal open subballs of the balls in $F_{\lambda_{0}}(x)$ and $\models\neg\phi(x,y;m)$. Hence $y\in\bigcup_{i}\points{F_{\mu_{i}}}(x)$. We may assume that $y\in\points{F_{\mu_{0}}}(x)$ and hence that $\points{F_{\mu_{0}}}(x)\cap\points{F_{\lambda}}(x)\neq\emptyset$. By Proposition\,\ref{prop:inter gen irr}, we must have $\points{F_{\mu_{0}}}(x)\cap\points{F_{\lambda}}(x) = \points{F_{\kappa}}(x)$ for both $\kappa = \lambda$ and $\kappa = \mu_{0}$, i.e. $F_{\lambda}(x) = F_{\mu_{0}}(x)$. Because such an equality is decided by $p$, this holds for all realisations of $p$.

It follows that $Y_{m,x} \subseteq \bigcup_{i} F_{\mu_{i}}(x)$ and therefore that $\card{Y_{m,x}} \leq r l < \infty$.
\end{proof}

Thus for all $(x,y)\models p$, only $k$ balls among the ones in $\bigcup_{i}F_{\mu_{i}}(x)$ cover $\phi(x,\points{{F}_{\lambda_{0}}}(x);m)$. As in Proposition\,\ref{prop:inter balls}, we may assume that for all $i$, $F_{\mu_{i}}(x)\subseteq\bigcup_{j=1}^{k} F_{\mu_{j}}(x)$. It follows that: \[p(x,y)\vdash \bigwedge_{j=1}^{k}\points{F_{\mu_{j}}}(x)\subset \points{F_{\lambda_{0}}}(x) \wedge (y\in\points{F_{\lambda_{0}}}(x)\sminus\bigcup_{i=1}^{k}\points{F_{\mu_{i}}}(x) \impform \phi(x,y;m))\]
where $k$ does not depend on $m$. We can now conclude as in Proposition\,\ref{prop:imp def strict open}.
\end{proof}

\begin{corollary}[imp def gen]
Assume $p$ is $\tL(\tA)$-quantifiable and $F$ has both the generic covering property and the maximal open subball property over $p$. Then $\Gen{\cE(\tM)}[p]$ is $\tL(\tA)$-quantifiable.
\end{corollary}

\begin{proof}
This follows immediately from Propositions\,\ref{prop:imp def strict open} and \ref{prop:imp def closed}. Indeed, either $\cE(\tM)$ is non empty and has no smallest element or it has a smallest element which consists of open balls or it has a smallest element which consists of closed balls. If it is empty, we could, equivalently, take $\cE$ to consist of all the $\lambda\in\Lambda$ such that $F_{\lambda}$ is constant equal to $\K$.
\end{proof}

Let us conclude this section by showing that, as previously, we can find families of balls verifying all the necessary hypotheses. Because both the generic covering property and the maximal open subball property are instances of being able to find large balls in a family, let us first consider the following definition. Recall that $d_{i}(B_{1},B_{2})$ is the $i$-th distance between balls of $B_{1}$ and balls of $B_{2}$ (see Definition\,\ref{def:di})

\begin{definition}(Generic large ball property)
We say that $F$ has the generic large ball property over $p$ if for all $\lambda_{1}$, $\lambda_{2}\in\Lambda(M)$ and $i\in\Nn$, there exists $(\mu_{j})_{0\leq j < l}\in\Lambda(M)$ such that:
\begin{thm@enum}
\item For all $j$, $p(x,y)\vdash$ "the balls in $F_{\mu_{j}}(x)$ are closed";
\item For all $j$, $p(x,y)\vdash \rad(F_{\mu_{j}}(x)) = d_{i}(F_{\lambda_{1}}(x),F_{\lambda_{2}}(x))$.
\item $p(x,y)\vdash \points{F_{\lambda_{1}}}(x) \subseteq \bigcup_{j}\points{F_{\mu_{j}}}(x)$;
\end{thm@enum}
and, if $p(x,y)\vdash \rad(F_{\lambda_{1}}(x)) < d_{i}(F_{\lambda_{1}}(x),F_{\lambda_{2}}(x))$ or $p(x,y)\vdash$ "the balls in $F_{\lambda_{1}}(x)$ are open", there exists $(\rho_{j})_{j <l}\in\Lambda(M)$ such that:
\begin{thm@enum}
\item For all $j$, $p(x,y)\vdash$ "the balls in $F_{\rho_{j}}(x)$ are open";
\item For all $j$, $p(x,y)\vdash \rad(F_{\rho_{j}}(x)) = d_{i}(F_{\lambda_{1}}(x),F_{\lambda_{2}}(x))$.
\item $p(x,y)\vdash \points{F_{\lambda_{1}}}(x) \subseteq \bigcup_{j}\points{F_{\rho_{j}}}(x)$;
\end{thm@enum}
\end{definition}

\begin{definition}(Good representation)
Let $\Delta(x,y;t)$ and $\Theta(x,y;s)$ be two finite sets of $\LL$-formulas where $x\in\K<n>$ and $(F_{\lambda})_{\lambda\in\Lambda}$ and $(G_{\omega})_{\omega\in\Omega}$ be two $\LL$-definable families of functions $\K<n>\to\Ballsst{l}$. We say that $(\Theta,G,x)$ is a good representation of $(\Delta,F,x)$ if for all $\LL(M)$-definable $p\in\TP[x]<\Theta>(M)$:
\begin{thm@enum}
\item\label{gr:adapted} $\Theta$ is adapted to $G$;
\item\label{gr:inter} $(G_{\omega})_{\omega\in\Omega_{p}}$ is closed under generic intersection over $p$;
\item\label{gr:large ball} $(G_{\omega})_{\omega\in\Omega_{p}}$ has the generic large ball property over $p$;
\item\label{gr:decide Delta} $p$ decides all formulas in $\Delta$;
\item\label{gr:cover F} For all $\lambda\in\Lambda(M)$, there exists a finite number of $(\omega_{i})_{0\leq i < l} \in\Omega_{p}(M)$ such that $p(x,y) \vdash F_{\lambda}(x) = \bigcup_{i} G_{\omega_{i}}(x)$.
\end{thm@enum}
where $\Omega_{p} := \{\omega\in\Omega\mid G_{\omega}$ is generically irreducible over $p\}$.
\end{definition}

If we only want to say that \ref{gr:adapted} to \ref{gr:large ball} hold we will say that $(\Theta,G,x)$ is a good representation.

\begin{proposition}[exists good rep](Existence of good representations)
Let $(F_{\lambda})_{\lambda\in\Lambda}$ be any $\LL$-definable family of functions $\K<n>\to\Ballsst{l}$ and $\Delta(x;t)$ any finite set of $\LL$-formulas where $x\in\K<n>$. Then, there exists a good representation $(\Psi,G,x)$ of $(\Delta,F,x)$.
\end{proposition}

\begin{proof}
Let us begin with some lemmas.

\begin{lemma}[exists large balls]
There exists $(H_{\rho})_{\rho\in\Rho}$ an $\LL$-definable family of functions $\K<n>\to\Ballsst{l}$ and $\Xi(x;t,s)\supseteq\Delta(x;t)$ a finite set of $\LL$-formulas adapted to $H$ such that $H$ has the generic large ball property over any $\Xi$-type and for all $\lambda\in\Lambda$, there exists $\rho\in\Rho$ such that $H_{\rho} = F_{\lambda}$.
\end{lemma}

\begin{proof}
For all $\lambda$, $\mu$, $\eta\in\Lambda$ and $i\leq l^{2}$, define $H_{\lambda,\mu,\eta,i,1}(x)$ to be the closed balls with radius $\min\{d_{i}(F_{\mu}(x),F_{\eta}(x)),\rad(F_{\lambda}(x))\}$ around the balls in $F_{\lambda}(x)$. If the balls in $F_{\lambda}(x)$ are open or if they are closed of radius strictly smaller than $d_{i}(F_{\mu}(x),F_{\eta}(x))$, define $H_{\lambda,\mu,\eta,i,0}(x)$ to be the set open balls with radius $d_{i}(F_{\mu}(x),F_{\eta}(x))$ around the balls in $F_{\lambda}(x)$. Otherwise, define $H_{\lambda,\mu,\eta,i,0}(x)$ to be the set of closed balls with radius $\min\{d_{i}(F_{\mu}(x),F_{\eta}(x)),\rad(F_{\lambda}(x))\}$ around the balls in $F_{\lambda}(x)$. By usual coding tricks, we may assume that $H$ is an $\LL$-definable family of functions. Adding finitely many formulas to $\Delta$ we obtain $\Xi(x;t,s)$ which is adapted to $H$. Let $p\in\TP[x]<\Xi>(M)$ and $x\models p$.

Let us first show the closed ball case of the generic large ball property. For all $\lambda_{k}$, $\mu_{k}$, $\eta_{k}\in\Lambda(M)$, $i_{k}$, $j_{k}\in\Nn$, for $k\in\{1,2\}$, and $r\in\Nn$, $d := d_{r}(H_{\lambda_{1},\mu_{1},\eta_{1},i_{1},j_{1}}(x),H_{\lambda_{2},\mu_{2},\eta_{2},i_{2},j_{2}}(x))$ is either the radius of the balls in $H_{\lambda_{k},\mu_{k},\eta_{k},i_{k},j_{k}}(x)$, i.e. $d_{i_{k}}(F_{\mu_{k}}(x),F_{\eta_{k}}(x))$ or $\rad(F_{\lambda_{k}}(x))$, or the distance between two disjoint balls from the $H_{\lambda_{k},\mu_{k},\eta_{k},i_{k},j_{k}}(x)$ in which case it is also the distance between some disjoint balls in the $F_{\lambda_{k}}(x)$. If $d = d_{i_{k}}(F_{\mu_{k}}(x),F_{\eta_{k}}(x))$, it is easy to check that $H_{\lambda_{1},\eta_{k},\mu_{k},i_{k},1}$ has all the suitable properties; and that this one instance suffices. Otherwise there exists some $m$ such that $H_{\lambda_{1},\lambda_{1},\lambda_{2},m,1}(x)$ is suitable.

The same reasoning applies to the open ball case (the extra conditions under which we have to work are just here to ensure that the balls in $F_{\lambda_{1}}(x)$ are indeed smaller than those we are trying to build around them).
\end{proof}

\begin{lemma}[large ball to refinement]
Assume that $F$ has the generic large ball property over any $\Delta$-type. Let $(G_{\omega})_{\omega\in\Omega}$ be any $\LL(M)$-definable family of functions $\K<n>\to\Ballsst{l}$ and $\Theta(x;s)$ be any finite set of $\LL$-formulas adapted to $G$ such that for all $p\in\TP[x]<\Theta>(M)$, we have:
\begin{thm@enum}
\item For all $\omega\in\Omega(M)$, there is $\lambda\in\Lambda(M)$ such that $p(x)\vdash G_{\omega}(x)\subseteq F_{\lambda}(x)$;
\item For all $\lambda\in\Lambda(M)$, there is $(\omega_{i})_{0\leq i <l}\in\Omega(M)$ such that $p(x)\vdash F_{\lambda}(x) = \bigcup_{i}G_{\omega_{i}}(x)$.
\end{thm@enum}
Then $G$ also has the generic large ball property over any $\Theta$-type.
\end{lemma}

\begin{proof}
Let $\omega_{1}$, $\omega_{2}\in\Omega(M)$, $i\in\Nn_{>0}$ and $x\models p$. Then there exists $\lambda_{1}$, $\lambda_{2}\in\Lambda(M)$ such that $G_{\omega_{k}}(x)\subseteq F_{\lambda_{k}}(x)$. Then $d_{i}(G_{\omega_{1}}(x),G_{\omega_{2}}(x))$ is either the radius of one of the balls involved and hence is the radius of one of $F_{\lambda_{k}}(x)$ or the distance between a ball in $G_{\omega_{1}}(x)$ and a ball in $G_{\omega_{2}}(x)$, i.e. the distance between a ball in $F_{\lambda_{1}}(x)$ and one in $F_{\lambda_{2}}(x)$. In both cases, the large closed ball property in $F$ allows us to find $(\mu_{j})_{0\leq j < l}\in\Lambda(M)$ such that $\points{G_{\omega_{1}}}(x)\subseteq \points{F_{\lambda_{1}}}(x)\subseteq \bigcup_{j}\points{F_{\mu_{j}}}(x)$, for all $j$, the balls in $F_{\mu_{j}}(x)$ are closed and their radius is $d_{i}(G_{\omega_{1}}(x),G_{\omega_{2}}(x))$. But, by hypothesis there are $(\rho_{j,k})_{0\leq k <l}\in\Omega(M)$ such that $F_{\mu_{j}}(x) = \bigcup_{k} G_{\rho_{j,k}}(x)$. By picking one $\rho_{j,k}$ per ball in $G_{\omega_{1}}(x)$, we see that $l$ of them are enough to cover $G_{\omega_{1}}(x)$ and we are done. The open ball case is proved similarly as the extra conditions hold for $G_{\omega_{1}}$ and $G_{\omega_{2}}$ if and only if they hold for $F_{\lambda_{1}}$ and $F_{\lambda_{2}}$.
\end{proof}

Adding them if we have to, we may assume that there is an instance of $F$ constant equal to $\emptyset$ and another constant one equal to $\{\K\}$. Let $(H_{\rho})_{\rho\in\Rho}$ and $\Xi$ be as in Lemma\,\ref{lem:exists large balls}. Let $(G_{\omega})_{\omega\in\Omega}$ and $\Theta(x;u)$ be as given by Proposition\,\ref{prop:exists F Delta nice} applied to $H$. Let $p\in\TP[x]<\Theta>(M)$. Then Conditions\,\ref{gr:adapted}, \ref{gr:decide Delta} and \ref{gr:cover F} hold. Condition\,\ref{gr:inter} also holds, by Corollary\,\ref{cor:inter gen irr cor}, and by Lemma\,\ref{lem:large ball to refinement} applied to $(G_{\omega})_{\omega\in\Omega_{p}}$, \ref{gr:large ball} also holds.
\end{proof}

\begin{proposition}
Let $(\Delta(x;t),(F_{\lambda})_{\lambda\in\Lambda},x)$ be a good representation and $p\in\TP[x]<\Delta>(M)$ be $\LL(M)$-definable. Then $F_{p} := (F_{\lambda})_{\lambda\in\Lambda_{p}}$ has the generic covering property and the maximal open subball property over $p$, where $\Lambda_{p} := \{\lambda\in\Lambda\mid F_{\lambda}$ is generically irreducible over $p\}$
\end{proposition}

\begin{proof}
Let $x\models p$, $\lambda_{1}$, $\lambda_{2}\in\Lambda_{p}(M)$ be such that $\points{F_{\lambda_{1}}}(x)\subset \points{F_{\lambda_{2}}}(x)$. By the generic large ball property, there exists $\mu_{j}\in\Lambda_{p}(M)$ such that the balls in $F_{\mu_{j}}(x)$ are open of radius $\rad(F_{\lambda_{2}}(x))$ and $\points{F_{\lambda_{1}}}(x)\subseteq\bigcup_{j}\points{F_{\mu_{j}}}(x)$. We have proved the maximal open subball property.

Let now $E\subseteq\Lambda_{p}(M)$ and $(\lambda_{i})_{0\leq i < k}\in\Lambda_{p}(M)$ be such that for all $\mu\in E$, $\points{F_{\lambda_{i}}}(x)\subset\points{F_{\mu}}(x)$. For any $\mu_{1}$, $\mu_{2}\in E$, if the balls in $F_{\mu_{1}}(x)$ are smaller than the balls in $F_{\mu_{2}}(x)$, by irreducibility, as $\points{F_{\mu_{1}}}(x)\cap\points{F_{\mu_{2}}}(x)\supseteq F_{\lambda_{0}}(x)\neq\emptyset$, we have $\points{F_{\mu_{1}}}(x)\subseteq \points{F_{\mu_{2}}}(x)$. Let us define the following equivalence relation on $\bigcap_{\lambda\in E} \points{F_{\lambda}}(x)$: $y_{1}\equiv y_{2}$ if for all $\mu\in E$, $y_{1}$ and $y_{2}$ are in the same ball from $F_{\mu}(x)$. For all non equivalent $y_{1}$ and $y_{2}$, there exists $\mu\in E$ such that $y_{1}$ and $y_{2}$ are not in the same ball from $F_{\mu}(x)$. This also holds for any $\eta\in E$ such that $\points{F_{\eta}}(x)\subseteq \points{F_{\mu}}(x)$. Thus there are at most $l$ equivalence classes and there exists $\mu_{0}\in E$ such that each equivalence class is contained in a different ball of $F_{\mu_{0}}(x)$.

Let $(P_{j})_{j\in J}$ denote these equivalence classes and $B_{j} = \{b\in\bigcup_{i} F_{\lambda_{i}}(x)\mid b\subseteq P_{j}\}$. The set $R_{j} := \{d(b_{1},b_{2})\mid b_{1},b_{2}\in B_{j}\}\cup\{\rad(b)\mid b\in B_{j}\}$ is finite and hence has a minimum $\gamma_j$. By the generic large ball property, there exists $\mu_{j}\in\Lambda_{p}(M)$ such that the balls in $F_{\mu_{j}}(x)$ are closed of radius $\gamma_j$ and one of its balls (call it $b_{0}$) contains one of the balls in $B_{j}$. In fact $b_{0}$ contains all of them as $\gamma_j$ is the minimum of $R_{j}$. For all $\kappa\in E$, all $b\in B_{j}$ are such that $b\subset\points{F_{\kappa}}(x)$. If $\rad(b_{0}) = d(b_{1},b_{2})$ for some some $b_{1}$, $b_{2}\in B_{j}$ then , because $b_{1}$ and $b_{2}$ are in the same ball from $F_{\kappa}(x)$, $\rad(b_{0}) = d(b_{1},b_{2})\leq \rad(F_{\kappa}(x))$. If $\rad(b_{0}) = \rad(b)$ for some $b\in B_{j}$, then because $b$ is inside one of the balls from $F_{\kappa}(x)$, $\rad(b_{0}) = \rad(b)\leq\rad(F_{\kappa}(x))$. In both cases, $b_{0}\subseteq\points{F_{\mu}}(x)$. Let $\eta_{j}$ be such that $\points{F_{\eta_{j}}}(x) = \points{F_{\mu_{j}}}(x)\cap \bigcap_{\kappa\in E}\points{F_{\kappa}}(x)$. Such an $\eta_{j}$ exists by generic intersection and because, by Proposition\,\ref{prop:inter balls}, this intersection is given by the intersection of a finite numbers of its elements. 

Then, as $F_{\eta_{j}}(x) \subseteq F_{\mu_{j}}(x)$, the balls in $F_{\eta_{j}}(x)$ are closed. Obviously, for all $\kappa\in E$, $\points{F_{\eta_{j}}}(x)\subseteq \points{F_{\kappa}}(x)$. Moreover, for all $i$, $\points{F_{\lambda_{i}}}(x)\subseteq \bigcup_{j} \points{F_{\mu_{j}}}(x)$ and for all $\kappa\in E$, $\points{F_{\lambda_{i}}}(x)\subseteq \points{F_{\kappa}}(x)$, hence we also have $\points{F_{\lambda_{i}}}(x)\subseteq \bigcup_{j} \points{F_{\eta_{j}}}(x)$. As there are at most $l$ of the $\eta_{j}$, we are done.
\end{proof}

\section{\texorpdfstring{$\Valgp$}{Gamma}-reparametrisations}
\label{sec:reparam}

Let $\LL\supseteq\Ldiv$, $T\supseteq\ACVF$ be an $\LL$-theory which eliminates imaginaries. Assume that $T$ is $C$-minimal. The two main examples of such theories are $\ACVFG$ and $\ACVFanneq$ where $\Ann$ is some separated Weierstrass system (for example $\bigcup_{m,n}\Zz[X_{0},\ldots,X_{n}][[Y_{0},\ldots,Y_{m}]]$) and $\ACVFann!$ denotes the theory of algebraically closed valued fields with $\Ann$-analytic structure (see \cite{CluLip-An} or \cite[Section\,3]{Rid-ADF}). This structure is considered in the language $\LAQ := \Ldiv\cup\Ann\cup\{^{-1}\}$.

\begin{remark}[EI Valgp]
The value group $\Valgp$ is stably embedded and $o$-minimal in $T$. As $\Valgp$ is an $o$-minimal group, the induced structure on $\Valgp$ eliminates imaginaries.
\end{remark}

\begin{proof}
Let $M\models T$ and $X\subseteq \Valgp$ be a unary $\LL(M)$-definable set. The set $\val^{-1}(X)$ is both a (potentially infinite) union of annuli around $0$ and a finite union of Swiss cheeses. Hence it is a finite union of annuli around $0$ and $X$ must be a finite union of intervals. Therefore, $\Valgp$ is $o$-minimal in $T$. By \cite{HasOns-EmbOmin}, $\Valgp$ is stably embedded in models of $T$.
\end{proof}

Let $M\models T$, $f = (f_{\lambda}:\K<n>\to\Valgp)_{\lambda\in\Lambda}$ be an $\LL(M)$-definable family of functions, $\Delta(x;t)$ be a finite set of $\LL$-formulas and $p\in\TP[x]<\Delta>(M)$. We wish to study the family $f$ and in particular its germs over $p$ (see Definition\,\ref{def:germ}), to show that they are internal to $\Valgp$. This is later used as a partial elimination of imaginaries result in enrichments $\tT$ of $T$ where $\Valgp$ is stably embedded: any subset of these germs definable in $\tT$ is coded in $\eq{\Valgp}$. The idea of the proof is to reparametrise the family of functions.

\begin{definition}[reparam]($\Valgp$-reparametrisation)
An $\LL(M)$-definable family $(g_\omega:\K<n>\to\Valgp)_{\gamma\in G}$, where $G\subseteq \Valgp<k>$ for some $k$, $\Valgp$-reparametrises $f$ over $p$ if for all $\lambda\in\Lambda(M)$, there is $\gamma\in G(M)$ such that \[p(x)\vdash f_{\lambda}(x) = g_{\gamma}(x).\]

An $\LL(M)$-definable family $(g_{\omega,\gamma}:\K<n>\to\Valgp)_{\omega\in\Omega,\gamma\in G}$ of functions $\K<n>\to\Valgp$, where $G\subseteq \Valgp<k>$ for some $k$, uniformly $\Valgp$-reparametrises $f$ over $\Delta$-types if for every $p\in\TP[x]<\Delta>(M)$ there exists $\omega_{0}\in\Omega(M)$ such that $g_{\omega} = (g_{\omega,\gamma})_{\gamma\in G}$ $\Valgp$-reparametrises $f$ over $p$.

We say that $T$ admits uniform $\Valgp$-reparametrisations if for every $\LL(M)$-definable family $f = (f_{\lambda})_{\lambda\in\Lambda}$ of functions $\K<n>\to\Valgp$ there exists a finite set of $\LL$-formulas $\Delta(x;s)$ and an $\LL(M)$-definable family $g = (g_{\omega,\gamma})_{\omega\in\Omega,\gamma\in G}$ of functions $\K<n>\to\Valgp$ which uniformly $\Valgp$-reparametrises $f$ over $\Delta$-types.
\end{definition}

We will say that $\Delta$ is adapted to $f$ (respectively to $g$) when any $\Delta$-type decides when $f_{\lambda_{1}}(x) = f_{\lambda_{2}}(x)$ (respectively $g_{\gamma_{1}}(x) = g_{\gamma_{2}}(x)$).

\begin{definition}[germ]($p$-germ)
Assume that $\Delta$ is adapted to $f$ and that $p$ is $\LL(M)$-definable. We say that $f_{\lambda_{1}}$ and $f_{\lambda_{2}}$ have the same $p$-germ if $p(x)\vdash f_{\lambda_{1}}(x) = f_{\lambda_{2}}(x)$. Let us denote $\germ!{p}{f_{\lambda}}\in M$ the code of the equivalence class of $\lambda$ under the equivalence relation "having the same $p$-germ".
\end{definition}

\begin{proposition}[germ reparam intern]
Let $g$ be a $\Valgp$-reparametrisation of $f$ over $p$, $\Delta$ be adapted to both $f$ and $g$ and $p$ be $\LL(M)$-definable. The set $\{\germ{p}{f_{\lambda}}\mid\lambda\in\Lambda\}$ is internal to $\Valgp$, i.e. there is an $\LL(M)$-definable one to one map from this set into some Cartesian power of $\Valgp$.
\end{proposition}

\begin{proof}
As $\Valgp$ is stably embedded in $T$ and eliminates imaginaries (see Remark\,\ref{rem:EI Valgp}), we may assume that $\germ{p}{g_{\gamma}}\in\Valgp$. Now pick any $\lambda$. Let $\gamma$ be such that $p(x)\vdash f_{\lambda}(x) = g_{\gamma}(x)$. Then $\germ{p}{g_{\gamma}}$ only depends on $\germ{p}{f_{\lambda}}$ and not on $\lambda$ or $\gamma$. It follows that the set $\{\germ{p}{f_{\lambda}}\mid\lambda\in\Lambda\}$ is in $\LL(M)$-definable one to one correspondence with a subset of the set $\{\germ{p}{g_{\gamma}}\mid \gamma\in G\}$ which is itself a subset of some Cartesian power of $\Valgp$.
\end{proof}

If $Z_{1}$ and $Z_{2} \subseteq \K$ are finite sets, we define $D(Z_{1},Z_{2}) := \{\val(z_{1}-z_{2})\mid z_{1}\in Z_{1}$ and $z_{2}\in Z_{2}\}$. Let us order the elements in $D(Z_{1},Z_{2})$ as $d_{1} > d_{2} > \cdots > d_{k}$ and let $d_{i}(Z_{1},Z_{2}) := d_{i}$. If $Z_{1} = \{z\}$ is a singleton we will write $d_{i}(z,Z_{2})$.

\begin{proposition}[reparam poly]
Let $t(x,y,\lambda) : \K<n+1+l> \to \K$ be an $\Sortrestr{\LL}{\K}(M)$-term polynomial in $y$, i.e. $t = \sum_{i=0}^{d}t_{i}(x,\lambda)y^{i}$, where $\card{x}= n$, $\card{y} = 1$ and $\card{\lambda} = l$. Let $Z_{\lambda}(x) := \{y\mid t(x,y,\lambda) = 0\}$. Then there exists an $\LL(M)$-definable family $q = (q_{\eta})_{\eta\in\Eta}$ of functions $\K<n>\to\Valgp$ such that for all $N\supsel M$, $x\in\K<n>(N)$ and $y\in\K(N)$, there exists $\mu_{0}\in \Lambda(M)$ such that for all $\lambda\in\Lambda(M)$ there exists $\eta\in\Eta(M)$ and $n$ smaller than the degree of $t$ in $y$ such that:
\[\val(t(x,y,\lambda)) = q_{\eta}(x) + n\cdot d_{1}(y,Z_{\mu_{0}}(x)).\]
\end{proposition}

\begin{proof}
Let us define $u(x,\lambda) := \frac{t(x,y,\lambda)}{\prod_{\alpha\in Z_{\lambda}(x)}(y-\alpha)^{m_{\alpha}}}$ where $m_\alpha$ denotes the multiplicity of $\alpha$. Let us also define $q_{\lambda,k,\uple{j},\eta}(x) := \val(u(x,\lambda)) + \sum_{i=0}^{k} d_{j_{i}}(Z_{\lambda}(x),Z_{\eta}(x))$
where $k$ is at most the degree of $t$ in $y$ and $j_{i}\leq l^{2}$. Note that because we can code disjunctions on a finite number of integers, $q$ can be considered as an $\LL(M)$-definable family of functions $\K<n>\to\Valgp$.

Let $N\supsel M$, $x\in\K<n>(N)$ and $y\in\K(N)$. First, assume that there exists $\mu_{0}\in\Lambda(M)$ and $\alpha_0\in Z_{\mu_{0}}(x)$ such that $\val(y-\alpha_{0}) = d_{1}(y,Z_{\mu_{0}}(x)) = \max_{\mu}\{d_{1}(y,Z_{\mu}(x))\}$. Now pick any $\lambda\in\Lambda(M)$ and $\alpha\in Z_{\lambda}(x)$.

\begin{claim}
Either $\val(y-\alpha) = d_{1}(y,Z_{\mu_{0}}(x))$ or $\val(y-\alpha) = d_{j_{\alpha}}(Z_{\lambda}(x),Z_{\mu_{0}}(x))$ for some $j_{\alpha}$.
\end{claim}

\begin{proof}
If $\val(y-\alpha) \neq d_{1}(y,Z_{\mu_{0}}(x))$, then $\val(y-\alpha) < d_{1}(y,Z_{\mu_{0}}(x))$. It follows that $\val(y-\alpha) = \val(\alpha - \alpha_{0}) = d_{j}(Z_{\lambda}(x),Z_{\mu_{0}}(x))$ for some $j$.
\end{proof}

Let $Z_{1} := \{\alpha\in Z_{\lambda}(x)\mid \val(y-\alpha) = d_{1}(y,Z_{\mu_{0}}(x))\}$ and $n := \sum_{\alpha\in Z_{1}}m_{\alpha}$. We have:
\[\begin{eqn}\val(t(x,y,\lambda)) &=& \val(u(x,\lambda)) + \sum_{\alpha\in Z_{\lambda}(x)}m_{\alpha}\val(y-\alpha)\\
&&\val(u(x,\lambda)) + \sum_{\alpha\nin Z_{1}}m_{\alpha}d_{j_{\alpha}}(Z_{\lambda}(x),Z_{\mu_{0}}(x)) + n \cdot d_{1}(y,Z_{\mu_{0}}(x))\\
&& q_{\lambda,k,\uple{j},\eta}(x) + n \cdot d_{1}(y,Z_{\mu_{0}}(x))
\end{eqn}\]
for some $k$ and $\uple{j}$.

If there does not exist a maximum in $\{d_{1}(y,Z_{\mu}(x))\}$, for any $\lambda\in\Lambda(M)$, then there exists $\eta\in\Lambda(M)$ and $\alpha_{0}\in Z_{\eta}(x)$ such that $val(y-\alpha_{0}) = d_{1}(y,Z_{\eta}(x)) > d_{1}(y,Z_{\lambda}(x))$. For all $\alpha\in Z_{\lambda}(x)$, $\val(y-\alpha) = \val(\alpha - \alpha_{0}) = d_{j_{\alpha}}(Z_{\lambda}(x),Z_{\eta}(x))$ for some $j_{\alpha}$. It follows that:
\[\begin{eqn}
\val(t(x,y,\lambda)) &=& \val(u(x,\lambda)) + \sum_{\alpha\in Z_{\lambda}(x)} m_{\alpha} d_{j_{\alpha}}(Z_{\lambda}(x),Z_{\eta}(x))\\
&=& q_{\lambda,k,\uple{j},\mu_{0}}(x)
\end{eqn}\]
for some $k$ and $\uple{j}$.
\end{proof}

\begin{proposition}[exists reparam]
Uniform $\Valgp$-reparametrisations exist in $\ACVFG$ and $\ACVFanneq$.
\end{proposition}

\begin{proof}
Let $f = (f_{\lambda})_{\lambda\in\Lambda}$ be an $\LL(M)$-definable family of  functions $\K<n>\to\Valgp$. We work by induction on $n$. The case $n=0$ is trivial as $f$ is nothing more than a family of points in $\Valgp$ that can be reparametrised by themselves. Let us now assume that $n = m+1$ and $x = (y,z)$ where $\card{z} = 1$. Because $\K$ is dominant, we may assume up to reparametrisation that $\lambda$ is a tuple from $\K$. If $T = \ACVFG$, the graph of $f_{\lambda}$ is given by an $\LG(M)$-formula. If $T = \ACVFanneq$, by \cite[Corollary\,5.5]{Rid-ADF} there exists an $\LG(M)$-formula $\psi(z,\uple{w},\gamma)$ and $\Sortrestr{\LL}{\K}$-terms $\uple{r}(x,\lambda)$ such that $M\models f_{\lambda}(y,z) = \gamma$ if and only if $M\models \psi(z,\uple{r}(y,\lambda),\gamma)$. Taking $\uple{r}$ to be the identity, the graph of $f_{\lambda}$ also has this form when $T = \ACVFG$. By elimination of quantifiers in $\ACVFG$ (or in the two sorted language), we know that $\psi(z,\uple{w},\gamma)$ is of the form $\chi((\val(P_{i}(z,\uple{w})))_{0\leq i < k},\gamma)$ where $\chi$ is an $\Sortrestr{\LG}{\Valgp}$-formula and $P_{i}\in\K(M)[Y,\uple{W}]$. We may also assume that $\chi$ defines a function $h:\Valgp^{k}\to\Valgp$.

Let $t_{i}(y,z,\lambda) = P_{i}(z,\uple{r}(y,\lambda))$ and $q_{i} = (q_{i,\eta})_{\eta\in\Eta_{i}}$ be an $\LL(M)$-definable family of functions $\K<m>\to\Valgp$ as in Proposition\,\ref{prop:reparam poly} with respect to $t_{i}$. By the usual coding tricks we may assume that there is only one family $q = (q_{\eta})_{\eta\in\Eta}$ such that for all $i$ and $\eta\in\Eta_{i}$ there exists $\epsilon\in\Eta$ such that $q_{i,\eta} = q_{\epsilon}$. By induction, there exists a uniform $\Valgp$-reparametrisation for $q$, i.e. there exist a finite set of $\LL$-formulas $\Xi(y;s)$ and an $\LL(M)$-definable family $(u_{\epsilon,\delta})_{\epsilon\in \Epsilon,\delta\in D}$ of functions $\K^{m}\to\Valgp$, where $D\subseteq \Valgp<l>$ for some $l$, such that for any $p\in\TP[y]<\Xi>(M)$, for some $\epsilon_{0}\in\Epsilon(M)$, $(u_{\epsilon_{0},\delta})_{\delta\in D}$ is a $\Valgp$-reparametrisation of $q$. Let $Z_{i,\lambda}(y) := \{z\mid P_{i}(y,z,\lambda) = 0\}$ and $g_{\epsilon,\uple{\mu},\uple{\delta},\uple{n}}(y,z) := h((u_{\epsilon,\delta_{i}}(x) + n_{i}\cdot d_{1}(z,Z_{i,\mu_{i}}(y)))_{0\leq i < k})$, $\phi_{\uple{n}} :=$"$f_{\lambda}(y,z) = g_{\epsilon,\mu,\uple{\delta},\uple{n}}(y,z)$" and $\Delta := \Xi\cup\{\phi_{\uple{n}}\mid \uple{n}\in\Nn\}$. For all $p\in\TP[y,z]<\Delta>(M)$, there exists $\epsilon_{0}\in\Epsilon(M)$ such that $(u_{\epsilon_{0},\delta})_{\delta\in D}$ $\Valgp$-reparametrises $q$ over $\Langrestr{p}{\Xi}$. Let $(y,z)\models p$. By Proposition\,\ref{prop:reparam poly} there exists a tuple $\uple{\mu}_{0}\in\Lambda(M)$ such that for all $\lambda\in\Lambda(M)$, there exists tuples $\uple{\eta}\in\Eta(M)$ and $\uple{n}$ such that $\val(t_{i}(y,z,\lambda)) = q_{\eta_{i}}(y) + n_{i}\cdot d_{1}(y,Z_{i,\mu_{0,i}}(x))$. As $y\models\Langrestr{p}{\Xi}$, there exists $\delta_{i}\in D(M)$ such that $q_{\eta_{i}}(y) = u_{\epsilon_{0},\delta_{i}}(y)$. Therefore,
\[\begin{eqn}
f_{\lambda}(y,z) &=& h((\val(t_{i}(y,z,\lambda)))_{0\leq i < k})\\
&=& h((u_{\epsilon_{0},\delta_{i}}(y) + n_{i}\cdot d_{1}(y,Z_{i,\mu_{0,i}}(x)))_{0\leq i < k})\\
&=& g_{\epsilon_{0},\uple{\mu_{0}},\uple{\delta},\uple{n}}(y,z).
\end{eqn}\]
Because $p$ decides such equalities, this holds in fact for all realisations of $p$. We have just shown that $(g_{\epsilon_{0},\uple{\mu_{0}},\uple{\delta},\uple{n},})_{\uple{\delta}\in D,\uple{n}\in\Nn}$ reparametrises $f$ over $p$. But because $\uple{\delta}$ is a tuple from $\Valgp$ and disjunctions on a finite number of bounded integers can be coded in $\Valgp$, it is in fact a $\Valgp$-reparametrisation.
\end{proof}

\begin{question}
Do uniform $\Valgp$-reparametrisations exist in all $C$-minimal extensions of $\ACVF$?
\end{question}

\section{Approximating sets with balls}
\label{sec:approx}

As before, let $\tL\supseteq\LL\supseteq\Ldiv$ be languages, $\Real$ be the set of $\LL$-sorts, $T\supseteq\ACVF$ be a $C$-minimal $\LL$-theory which eliminates imaginaries and admits $\Valgp$-reparametrisations, $\tT$ an $\tL$-theory containing $T$, $\tN\models \tT$, $N := \Langrestr{\tN}{\LL}$ and $\tA = \acltLeq(\tA)\subseteq\tNeq$. Let us assume that $\res$ and $\Valgp$ are stably embedded in $\tT$ and that the induced theories on $\res$ and $\eq{\Valgp}$ eliminate $\exists^{\infty}$.

In this section we bring together all the work we have done in Sections\,\ref{sec:type ball}, \ref{sec:imp def} and \ref{sec:reparam} to construct definable types, in order to prove Theorem\,\ref{thm:dens def}. The core of the work is done in Lemma\,\ref{lem:ind approximate} where we show that we can enrich a quantifiable partial $\tL$-type with formulas of the form $y\in F_{\lambda}(x)$, where $F_{\lambda}$ is an $\LL$-definable family of functions $\K^{n}\to\Ballsst{l}$, while maintaining consistency with a given $\tL$-definable set. Once this is done, it is only a question of proving the various reductions sketched in the introduction. In Proposition\,\ref{prop:ind approx}, we show that we can enrich a quantifiable partial $\tL$-type with arbitrary formulas while maintaining consistency with a given $\tL$-definable set. Finally, in Proposition\,\ref{prop:approximate}, we show that every strict $(\tL,\star)$-definable set $X$ (see Definition\,\ref{def:strict star-def}) is consistent with a definable $\LL$-type.

Note that, even though all the types which are constructed in this section are $\LL$-types (or $\Delta$-types for some set $\Delta$ of $\LL$-formulas), they are definable using $\tL(\tN)$-formulas: for every $\phi(x;t)\in\Delta$, there exists an $\tL(\tN)$-formula $\defsc{p}{x}\phi(x;t)$ such $\phi(x;m)\in p$ if and only if $\tN\models\defsc{p}{x}\phi(x;m)$. One of the goals of \cite{RidSim-NIP} is to show that, under some more hypotheses, such types are indeed $\LL(N)$-definable.

\begin{lemma}[ind approximate]
Let $Y\subset\K^{n+1}$ be an $\tLeq(\tA)$-definable set and $(\Delta(x,y;t),(F_{\lambda})_{\lambda\in\Lambda},x)$ be a good representation where $x\in\K^{n}$. Let $p(x,y)\in\TP[x,y]<\Delta>(N)$ be $\tLeq(\tA)$-quantifiable (as a partial $\tLeq$-type) and consistent with $Y$. Assume that there exists an $\LL(N)$-definable family $g=(g_{\gamma})_{\gamma\in G}$ of functions $\K^{n}\to\Valgp$ which $\Valgp$-reparametrises the family $(\rad\comp F_{\lambda})_{\lambda\in\Lambda}$ over $p$.

Then there exists a type $q(x,y)\in\TP[x,y]<\Funform{\Delta}{F}>(N)$ which is $\tLeq(\tA)$-quantifiable and consistent with $p$ and $Y$.
\end{lemma}

We are looking for a type $q=\Gen{E}[p]$ so most of the work consists in finding the right $E$.

\begin{proof}
Let $\Lambda_{p} := \{\lambda\in\Lambda\mid F_{\lambda}$ is generically irreducible over $p\}$. We define a preorder $\Ordeq$ on $\Lambda_{p}$ by \[\lambda \Ordeq \mu\text{ if and only if }p(x,y)\vdash (y\in\points{F_{\lambda}}(x)\wedge (x,y)\in Y) \impform y\in\points{F_{\mu}}(x).\] By $\tLeq(\tA)$-quantifiability of $p$, $\Ordeq$ is $\tLeq(\tA)$-definable. Let $\sim$ be the associated equivalence relation, i.e. $\lambda \sim \mu$ if and only if $p(x,y)\vdash (\points{F_{\lambda}}(x)\wedge (x,y)\in Y) \iffform (y\in\points{F_{\mu}}(x)\wedge (x,y)\in Y)$. The preorder $\Ordeq$ induces an order on $\Lambda_{p}/{\sim}$ that we will also denote $\Ordeq$. We denote by $\widehat{\lambda}\subseteq\Lambda_{p}$ the $\sim$-class of $\lambda$. The set $\Lambda_{p}/{\sim}$ has a greatest element, $\widehat{\K}$, given by the class of any $\lambda\in\Lambda_{p}$ such that $F_{\lambda}$ is constant equal to $\{\K\}$. It also has a smallest element, $\widehat{\emptyset}$, given by the class of any $\lambda\in\Lambda_{p}$ such that $F_{\lambda}$ is constant equal to $\emptyset$. Because $p$ is consistent with $Y$, $\widehat{\K}\neq \widehat{\emptyset}$.

\begin{claim}[tree]
Let $\lambda\in\Lambda_{p}\sminus\widehat{\emptyset}$, then $\Ordeq$ totally orders $\{\widehat{\mu} : \mu\in\Lambda_{p}\wedge\lambda\Ordeq\mu\}$.
\end{claim}

\begin{proof}
Let $\mu_{1}$, $\mu_{2}\in\Lambda_{p}(N)$ such that $\lambda\Ordeq\mu_{i}$. Because $\lambda\nin\widehat{\emptyset}$ there exists $(x,y)\models p$ such that $y\in \points{F_{\lambda}}(x) \wedge (x,y)\in Y$.  As $\lambda\Ordeq\mu_{i}$, we also have $y\in\points{F_{\mu_{i}}}(x)$. Hence $\points{F_{\mu_{1}}}(x) \cap \points{F_{\mu_{2}}}(x)\neq\emptyset$. By Proposition\,\ref{prop:inter gen irr}, we may assume $\points{F_{\mu_{1}}}(x) \subseteq \points{F_{\mu_{2}}}(x)$. Then $\mu_{1}\Ordeq\mu_{2}$.
\end{proof}

Hence $((\Lambda_{p}/{\sim})\sminus\{\widehat{\emptyset}\},\Ordeq)$ is a tree with the root on the top. Let us now show that the branches of this tree are internal to $\Valgp$. Let $h(\lambda) := \germ{p}{\grad(F_{\lambda})}$. By Proposition\,\ref{prop:germ reparam intern}, we may assume (after adding some parameters) that the image of $h$ is in some Cartesian power of $\Valgp$. Let us also define $\push{h} : \widehat{\lambda}\mapsto \code{h(\widehat{\lambda})}$. By stable embeddedness of $\Valgp$, $\push{h}$ takes its values in $\eq{\Valgp}$.

\begin{claim}[internal branch]
Pick any $\lambda\in\Lambda_{p}\sminus\widehat{\emptyset}$, then the function $\push{h}$ is injective on $\{\widehat{\mu} : \lambda\Ordeq\mu\}$.
\end{claim}

\begin{proof}
Let $\mu_{1}$ and $\mu_{2}$ be such that $\lambda\Ordeq\mu_{i}$. We have seen in Claim\,\ref{claim:tree}, that we may assume that $p(x,y)\vdash\points{F_{\mu_{1}}}(x) \subseteq \points{F_{\mu_{2}}}(x)$. Let $(x,y)\models p$. If $\widehat{\mu_{1}} \neq \widehat{\mu_{2}}$ then we must have $\points{F_{\mu_{1}}}(x) \subset \points{F_{\mu_{2}}}(x)$. In particular, $\grad(F_{\mu_{1}}(x)) < \grad(F_{\mu_{2}}(x))$ and $h(\mu_{1})\neq h(\mu_{2})$. In fact, we have just shown that for all $\omega_{i}\in\widehat{\mu_{i}}$, $h(\omega_{1})\neq h(\omega_{2})$. Hence $\push{h}(\widehat{\mu_{1}})\neq \push{h}(\widehat{\mu_{2}})$.
\end{proof}

Let $\lambda\in\Lambda_{p}(N)$ be such that $\code{\widehat{\lambda}}\in \tA$. If $\Gen{\widehat{\lambda}(\tN)}[p]$ is consistent with $Y$, it is, in particular, consistent and consistent with $p$. By Proposition\,\ref{prop:Gen compl}, it is a complete $\Funform{\Delta}{F}$-type. By Corollary\,\ref{cor:imp def gen}, $\Gen{\widehat{\lambda}(\tN)}[p]$ is $\tLeq(\tA)$-quantifiable. It follows that taking $q = \Gen{\widehat{\lambda}(\tN)}[p]$ works. Therefore, it suffices to find a $\lambda\in\Lambda_{p}(N)$ such that  $\code{\widehat{\lambda}}\in\tA$ and $\Gen{\widehat{\lambda}(\tN)}[p]$ is consistent with $Y$.

\begin{claim}[incons next]
Let $\lambda\in\Lambda_{p}(N)$. If $\widehat{\lambda}\neq\widehat{\emptyset}$ and $\Gen{\widehat{\lambda}(\tN)}[p]$ is not consistent with $Y$, then there exists $\mu$ such that $\widehat{\mu}$ is an immediate $\Ordeq$-predecessor of $\widehat{\lambda}$ and $\code{\widehat{\mu}}\in \acltLeq(\tA\code{\widehat{\lambda}})$.
\end{claim}

\begin{proof}
As $\Gen{\widehat{\lambda}(\tN)}[p]$ is not consistent with $Y$, there exists $(\mu_{i})_{0<i<k}\in\Lambda_{p}(N)$ such that $\mu_{i}\Ord\lambda$ and $p(x,y)\vdash y\in\points{F_{\lambda}}(x)\wedge (x,y)\in Y \impform y\in \bigcup_{i=1}^{k}\points{F_{\mu_{i}}}(x)$. We may assume that for all $i$, $\mu_{i}\nin\widehat{\emptyset}$ and that $p(x,y)\vdash \points{F_{\mu_{i}}}(x)\cap \points{F_{\mu_{j}}}(x) = \emptyset$ for all $i\neq j$. Let $\kappa\in\Lambda_{p}(N)$ be such that $\mu_{i_{0}}\Ordeq\kappa\Ordeq \lambda$ for some $i_{0}$. Because $\mu_{i_{0}}\Ordeq\kappa$, we have $p(x,y)\vdash \points{F_{\kappa}}(x)\cap\points{F_{\mu_{i_{0}}}}(x)\neq \emptyset$. If $p(x,y)\vdash \points{F_{\kappa}}(x) \subseteq \points{F_{\mu_{i_{0}}}}(x)$ then $\kappa\Ordeq\mu_{i_{0}}$ and hence $\kappa \sim \mu_{i}$.

Also, as $\kappa\Ordeq\lambda$, we have $p(x,y)\vdash (y\in\points{F_{\kappa}}(x)\wedge (x,y)\in Y)\impform y\in \points{F_{\lambda}}(x) \impform y\in \points{F_{\mu_{i}}}(x)$. For any $i\neq i_{0}$, if $p(x,y)\vdash \points{F_{\kappa}}(x)\cap\points{F_{\mu_{i}}}(x)\neq \emptyset$ then we must have $p(x,y)\vdash \points{F_{\mu_{i}}}(x) \subseteq \points{F_{\kappa}}(x)$. Therefore, we have $p(x,y)\vdash (y\in\points{F_{\kappa}}(x)\wedge (x,y)\in Y) \iffform \bigvee_{i\in I}(y\in\points{F_{\mu_{i}}}(x)\wedge (x,y)\in Y)$, where $I = \{i\mid \points{F_{\mu_{i}}}(x)\cap \points{F_{\kappa}(x)}\neq\emptyset\}$. It follows that the set $\{\widehat{\kappa}\mid \mu_{i}\Ordeq\kappa\Ordeq\lambda$ for some $i\}$ is finite. In particular we could choose $\mu_{i}$ such that there is no $\kappa$ such that $\widehat{\mu}_{i} \Ord \widehat{\kappa}\Ord \widehat{\lambda}$. The $\widehat{\mu}_{i}$ are the (finitely many) direct $\Ordeq$-predecessors of $ \widehat{\lambda}$ and therefore $\widehat{\mu}_{i}\in \acltLeq(\tA\code{\widehat{\lambda}})$.
\end{proof}

Let us assume that there does not exist $\lambda$ such that $\code{\widehat{\lambda}}\in\tA$ and $\Gen{\widehat{\lambda}(\tN)}[p]$ is consistent with $Y$. Starting with  $\widehat{\lambda_{0}} = \widehat{\K} \in\tA$, we construct, using Claim\,\ref{claim:incons next}, a sequence $(\lambda_{i})_{i\in\omega}$ such that $\widehat{\lambda}_{i+1}$ is a direct $\Ordeq$-predecessor of $\widehat{\lambda}_{i}$. For all $i$, we have $\card{\{\widehat{\mu}\mid \widehat{\lambda_{i}}\Ordeq\widehat{\mu}\}} = i+1 = \card{\push{h}(\{\widehat{\mu}\mid \widehat{\lambda_{i}}\Ordeq\widehat{\mu}\})}$, contradicting the elimination of $\exists^{\infty}$ in $\eq{\Valgp}$. This concludes the proof
\end{proof}

\begin{proposition}[ind approx]
Let $Y\subseteq \K<n+m>$ be an $\tLeq(\tA)$-definable set and $\Delta(x,y;t)$ and $\Theta(y;s)$ be finite sets of $\LL$-formulas where $\card{x} = n$ and $\card{y} = m$. Let $p\in\TP[x,y]<\Delta>(N)$ be $\tLeq(\tA)$-quantifiable and consistent with $Y$. Then there exists a finite set of $\LL$-formulas $\Xi(x,y;s,t,r)\supseteq\Delta\cup\Theta$ and a type $q\in\TP[x,y]<\Xi>(N)$ which is $\tLeq(\tA)$-quantifiable and consistent with $p$ and $Y$.
\end{proposition}

\begin{proof}
We proceed by induction on $\card{y}$. The case $\card{y} = 0$ is trivial. Let us now assume that $y = (z,w)$ where $\card{w} = 1$. By Proposition\,\ref{prop:exists F Delta} there exists $\Phi(z;u)$ a finite set of $\LL$-formulas and $F = (F_{\lambda})_{\lambda\in\Lambda}$ an $\LL$-definable family of functions $\K^{m-1}\to \Balls{l}$ such that $\Funform{\Phi}{F}$ decides any formula in $\Theta$. By Propositions\,\ref{prop:F st balls} and \ref{prop:exists good rep} we can assume that the pair $F_{\lambda} : \K^{m-1}\to\Ballsst{l}$ and $(\Phi,F,z)$ is a good representation. We can make $F$ into an $\LL$-definable family of functions $\K<n+m-1>\to\Ballsst{l}$ by setting $G_{\lambda}(x,z) = F_{\lambda}(z)$. As $T$ admits $\Valgp$-reparametrisations, there exists $\Upsilon(x,z;v)$ such that for any $p\in\TP[y]<\Upsilon>(N)$, there exists a $\Valgp$-reparametrisation $(g_{\gamma})_{\gamma}$ of $(\rad\comp G_{\lambda})_{\lambda\in\Lambda}$ over $p$.

By induction applied to $\Delta(x,z,w;t)$, $\Phi(z;u)\cup\Upsilon(z;v)$ and $p$, we obtain a finite set of $\LL$-formulas $\Omega(x,w,z;r) \supseteq \Delta\cup\Phi\cup\Upsilon$ and a type $q_{1}\in\TP[x,z,w]<\Omega>(N)$ which is $\tLeq(\tA)$-quantifiable and consistent with $p$ and $Y$. We can now apply Lemma\,\ref{lem:ind approximate} to $Y$, $(\Omega,G,(x,z))$, $q_{1}$ and $g$ to find a type $q_{2}\in\TP[x,w,z]<\Funform{\Omega}{G}>(N)$ which is $\tLeq(\tA)$-quantifiable and consistent with $q_{1}$ and $Y$. As all the formulas in $\Theta$ are decided by $\Funform{\Omega}{G}$, we may assume that $q_{2}$ is in fact a $(\Funform{\Omega}{G}\cup\Theta)$-type. Then $\Xi = \Funform{\Omega}{G}\cup\Theta$ and $q = q_{2}$ are suitable.
\end{proof}

\begin{proposition}[approximate]
Let $X$ be non empty strict $(\tLeq(\tA),x)$-definable. Let $\Delta(x;t)$ be a countable set of $\LL$-formulas. Then there exists an $\tLeq(\tA)$-definable type $p\in\TP[x]<\Delta>(N)$ consistent with $X$.
\end{proposition}

\begin{proof}
We may assume that \(X\subseteq \K<n>\) for some \(n\in\Nn\). Let $\{\phi_{j}(x_{j};t_{j})\mid j<\omega\}$ be an enumeration of all formulas in $\Delta$ where $|x_{j}| < \infty$. Let $\Delta_{-1} := \emptyset$ and $p_{-1} := \emptyset$. We construct, for all $j$, a finite set $\Delta_{j}(x_{\leq j};s_{j})$ of $\LL$-formulas and a type $p_{j} \in \TP[\leq x_{j}]<\Delta_{j}>(N)$ such that for all $j<\omega$, $\Delta_{j}\cup\{\phi_{j}\} \subseteq \Delta_{j+1}$, $p_{j+1}$ is $\tLeq(\tA)$-quantifiable and consistent with $p_{j}$ and $X$. Let us assume that $p_{j}$ and $\Delta_{j}$ have been constructed. Let $Y_{j+1}$ be the projection of $X$ on the variables $x_{\leq j+1}$. Then $Y_{j+1}$ is $\tLeq(\tA)$-definable. We can then apply Proposition\,\ref{prop:ind approx} to $\Delta_{j}(x_{\leq j};s_{j})$, $\{\phi(x_{j+1};t_{j+1})\}$, $p_{j}$ and $Y_{j+1}$ in order to obtain $p_{j+1}$. As $Y_{j+1}$ is the projection of $X$ on the variables which appear in $p_{j}$ and $p_{j+1}$, and that $p_{1}$, $p_{j+1}$ and $Y$ are consistent, it follows that $p_{j}$, $p_{j+1}$ and $X$ are also consistent. We can now take $p := \bigcup_{j<\omega} p_{j}$. As each $p_{j}$ is $\tLeq(\tA)$-definable (as a $\Delta_{j}$-type), so is $p$ and thus $\restr{p}{\Delta}$.
\end{proof}

We now prove the main result we have been aiming for.

\begin{theorem}[dens def]
Let $\tL\supseteq\LL\supseteq\Ldiv$ be languages, $\Real$ be the set of $\LL$-sorts, $T\supseteq\ACVF$ be a $C$-minimal $\LL$-theory which eliminates imaginaries and admits $\Valgp$-reparametrisations. Let $\tT$ a complete $\tL$-theory containing $T$ such that $\K$ is dominant in $\tT$ and:
\begin{thm@enum}
\item\label{hyp:exists infty} The sets $\res$ and $\Valgp$ are stably embedded in $\tT$ and the induced theories on $\res$ and $\eq{\Valgp}$ eliminate $\exists^{\infty}$;

\item\label{hyp:QE} For any $\tN\models\tT$, $A = \K(\dcl[\tL](A)) \subseteq \tN$ and any $\tL(A)$-definable set $X\subseteq\K^{n}$, there exists an $\tL$-definable bijection $f : \K^{n} \to Y$ such that $f(X) = Y\cap Z$ where $Z$ is $\LL(A)$-definable; note that $f$ has to be defined without parameters.
\end{thm@enum}
Then for any $\tN\models\tT$, any countable set $\Delta(x;t)$ of $\tL$-formulas and any non empty $\tL(\tN)$-definable set $X(x)$, there exists $p\in\TP<\Delta>(\tN)$ which is consistent with $X$ and $\tLeq(\acltLeq(\code{X}))$-definable.

If, moreover, the following holds:
\begin{thm@enum}[2]
\item\label{hyp:nice model} There exists $\tM\models\tT$ such that $\Langrestr{\tM}{\LL}$ is uniformly stably embedded in every elementary extension;
\end{thm@enum}
then, the type $p$ can be assumed to be $\tL(\Real(\acltLeq(\code{X})))$-definable
\end{theorem}

\begin{proof}
Let $\tA := \acltLeq(\code{X})$. We may assume that $X\subset\K^{n}$ for some $n$. Indeed, let $S_{i}$ be the sorts such that $X\subseteq \prod S_{i}$.  Since $\K$ is dominant, there is an $\tL$-definable surjection $\pi :\K<n> \to \prod S_{i}$. If we find $p$ consistent with $Y := \pi^{-1}(X)$ and $\tL(\acltLeq(\code{Y}))$-definable, then $\push{\pi}[p]$ is consistent with $X$ and $\tL(\tA)$-definable. So we may assume that $X\subseteq \K<n>$

Let $F := \{f$ is an $\tL$-definable bijection whose domain is $\K^{n}\}$ and $\prol[\omega](x) := (f(x))_{f\in F}$. Then $\prol[\omega](X)$ is strict $(\tLeq(\tA),\star)$-definable. Pick any $\phi(x;t)\in\Delta(x;t)$. As $\K$ is dominant we may assume $t$ is a tuple of variables from $\K$ too. By (ii), for all tuples $m\in\K(\tN)$, there exists $(f : \K^{n}\to Y) \in F$ and an $\tL$-definable map $g$ (into $\K<l>$ for some $l$) such that $f(\phi(\tN;m)) = Y(\tN)\cap Z(\tN)$ where $Z$ is $\LL(g(m))$-definable. As $\tN$ is arbitrary, we may assume that it is sufficiently saturated and, by compactness, there exists a finite number of $(f_{i}:\K^{n}\to Y_{i})\in F$, $\tL$-definable maps $g_{i}$ and $\LL$-formulas $\psi_{i}(y_{i};s_i)$ such that for any tuple $m\in\K(\tN)$ there exists $i_{0}$ such that $f_{i_{0}}(\phi(\tN;m)) = \psi_{i_{0}}(\tN;g_{i_{0}}(m))\cap Y_{i_{0}}(\tN)$. Let $\Theta(y;s)$ be the (countable) set of all $\psi_i(y_{i};s_i)$ that can appear for a $\phi(x;t)\in\Delta(x;t)$.

By Proposition\,\ref{prop:approximate}, there exists an $\tL(\tA)$-definable type $p\in\TP[y]<\Theta>(N)$ consistent with $\prol[\omega](X)$. Let $q = \{x\mid \prol[\omega](x)\models p\}$. Then $q$ is consistent with $X$. There remains to show that it is a complete $\Delta$-type and that it is $\tL(\tA)$-definable. Pick $\phi(x;t)\in \Delta(x;t)$. Let $f_i$, $g_i$, $\psi_i$, $m$ and $i_0$ be as above. Let $c_{1}$, $c_{2}\models q$. Assume that $\models \phi(c_{1};m)$. Then $f_{i_{0}}(c_{1}) \in \psi_{i_{0}}(\tN;g_{i_{0}}(m))\cap Y_{i_{0}}(\tN)$. As $\prol[\omega](c_{1})$ and $\prol[\omega](c_{2})$ have the same $\Theta(y;s)$-type over $\tN$ and $f_{i_{0}}(c_{2})\in Y_{i_{0}}(\tN)$, we also have $f_{i_{0}}(c_{2}) \in  \psi_{i_{0}}(\tN;g_{i_{0}}(m))\cap Y_{i_{0}}(\tN) = f_{i_{0}}(\phi(\tN;m))$. Because $f_{i_{0}}$ is a bijection, $\models \phi(c_{2};m)$. As for definability, we have just shown that $\phi(x;m)\in q$ if and only if $\psi_{i_{0}}(y_{i};g_{i_{0}}(m)) \in p$ for some $i_{0}$ such that $f_{i_{0}}(\phi(\tN;m)) = \psi_{i_{0}}(\tN;g_{i_{0}}(m))\cap Y_{i_{0}}(\tN)$ but that can be stated with an $\tL(\tA)$-formula.

If Hypothesis\,\ref{hyp:nice model} holds,by \cite[Corollary\,1.7]{RidSim-NIP} that the type $q$ is $\LL(\Real(\tA))$-definable and hence, $p$ is $\tL(\Real(\tA))$-definable.
\end{proof}

\begin{question}
Can the restriction on the cardinality of $\Delta$ be lifted to obtain the density of complete definable $\tL$-types even when $\tL$ is not countable?
\end{question}

The main problem is to prove Proposition\,\ref{prop:approximate} without any cardinality assumption on $\Delta$. The present proof relies on a induction that cannot be carried out beyond $\omega$ because the union of quantifiable types might not be quantifiable.

\begin{corollary}[VDF dens def]
Let $M\models \VDF$. Any $\LGD(M)$-definable set $X$ is consistent with an $\LG(\Geom(\acleq(\code{X})))$-definable $p\in\TP(M)$.
\end{corollary}

\begin{proof}
This follows from Theorem\,\ref{thm:dens def}, taking $T$ to be $\ACVFG$ and $\tT$ to be $\VDFG$. The fact that $\ACVFG$ admits $\Valgp$-reparametrisations is proved in Proposition\,\ref{prop:exists reparam}.

Hypothesis\,\ref{hyp:exists infty} follows from Theorem\,\ref{str Valgp} and \ref{str res} and the fact that both $\DCF[0]$ and $\DOAG$ eliminate $\exists^{\infty}$. Hypothesis\,\ref{hyp:QE} is an easy consequence of elimination of quantifiers: let $\phi(x;s)$ be an $\LGD$-formula such that $x$ and $s$ are tuples of field variables, then there exists an $\Ldiv$-formula $\psi(u;t)$ and $n\in\Nn$ such that $\phi(x;s)$ is equivalent modulo $\VDF$ to $\psi(\prol[n](x);\prol[n](s))$, i.e. for all $m\in\tN$, $\prol[n]$ is an $\LDdiv$-definable bijection between $\phi(\tN;m)$ and $\psi(x,\prol[n](m))\cap\prol[n](\K<\card{x}>)$. Hypothesis\,\ref{hyp:nice model} follows from the fact that if $k\models\DCF[0]$ then the Hahn field $k((t^{\Rr}))$ (with the derivation described in Example\,\ref{ex:Hahn VDF}), is a model of $\VDF$. By Corollary\,\ref{cor:Hahn unif stab emb} the underlying valued field is uniformly stably embedded in every elementary extension.
\end{proof}

\section{Imaginaries and invariant extensions}
\label{sec:EI IE}

In this section, we investigate the link between the density of definable types, elimination of imaginaries and the invariant extension property (see Definition\,\ref{def:inv ext}). I am very much indebted to \cite{Hru-EIACVF,Joh-EIACVF} for making me realise that the density of definable types could play an important role in proving elimination of imaginaries. To be precise, we will show that both the elimination of imaginaries and the invariant extension property follow from the density of types invariant over real parameters.

In the following proposition, we show that the density of $\Delta$-types invariant over real parameters for finite $\Delta$ suffices to prove weak elimination of imaginaries.

\begin{proposition}[EI crit]
Let $T$ be an $\LL$-theory and $\Real$ a set of its sorts such that for all $N\models T$, all non empty $\LL(N)$-definable sets $X$ and all $\LL$-formulas $\phi(x;s)$ (where $x$ is sorted as $X$), there exists $p\in\TP[x]<\phi>(N)$ which is consistent with $X$ and $\aut(N)[\Real(\acleq(\code{X}))]$-invariant. Then $T$ weakly eliminates imaginaries up to $\Real$.
\end{proposition}

\begin{proof}
Let $M$ be a sufficiently saturated and homogeneous model of $T$, $E$ be any $\LL$-definable equivalence relation, $X$ be one of its classes in $M$, $\phi(x,y)$ be an $\LL$-formula defining $E$ and $A =  \Real(\acltLeq(\code{X}))$. By hypothesis, there exists an $\aut(N)[A]$-invariant type $p\in\TP[x]<\phi>(M)$ consistent with $X$. Because $X$ is defined by an instance of $\phi$, we have in fact $p(x)\vdash x\in X$. For all $\sigma\in\aut(N)[A]$, $\sigma(X)$ is another $E$-class and $\sigma(p) = p \vdash x\in X$. It follows that $\sigma(X)\cap X\neq \emptyset$ and $X = \sigma(X)$. Therefore $\code{X}\in\dcleq(A) = \dcleq(\Real(\acltLeq(\code{X})))$, i.e. $X$ is weakly coded in $\Real$.
\end{proof}

Let us now consider the invariant extension property.

\begin{proposition}[inv dens ext]
Let $T$ be an $\LL$-theory, $A\subseteq M$ for some $M\models T$. The following are equivalent:
\begin{thm@enum}
\item For all $\LL(A)$-definable non empty sets $X(x)$, $\phi(x;s)$ an $\LL$-formula and $N\models T$, $A\subseteq $, there exists $p\in\TP[x]<\phi>(N)$ such that $p$ is $\aut(N)[A]$-invariant and consistent with $X$;
\item $T$ has the invariant extension property over $A$.
\end{thm@enum}
\end{proposition}

\begin{proof}
 Let us first show that (ii) implies (i). Let $N\models T$, $X(x)$ be an $\LL(A)$-definable non empty set, $\phi(x;s)$ an $\LL$-formula and $p\in\TP[x](A)$ be any type containing $X$. Let $q\in\TP[x](N)$ be an $\aut(N)[A]$-invariant extension of $p$. Then $\restr{q}{\phi}$ is consistent with $X$.

Conversely, let $\Theta = \{\phi(x;a)\iffform\phi(x;b)\mid a,b\in N\text{ and }\tp(a)[A] = \tp(b)[A]\}$. Then $q\in\TP[x](N)$ is invariant if and only if $\Theta\subseteq q$. Pick any $p\in\in\TP[x](A)$. Hypothesis (i) exactly implies that every finite subset of $p\cup\Theta$ is consistent, so, by compactness, there exists $c\models p\cup\Theta$, then $q = \tp(a)[N]$is an invariant extension of $p$.
\end{proof}

\begin{theorem}[EI IE crit]
In the setting of Theorem\,\ref{thm:dens def}, $\tT$ eliminates imaginaries and has the invariant extension property.
\end{theorem}

\begin{proof}
Weak elimination of imaginaries follows from Proposition\,\ref{prop:EI crit} and the invariant extension property follows from Proposition\,\ref{prop:inv dens ext}. In both cases the assumption on density of invariant $\phi$-types follows from Theorem\,\ref{thm:dens def}. Elimination of imaginaries then follows as any finite set in $\Real$ is also definable in $T$ and hence are coded in $T$.
\end{proof}

\begin{corollary}[VDF EI IE]
The theory $\VDFG$ eliminates imaginaries and has the invariant extension property.
\end{corollary}

\begin{proof}
This follows from Theorem\,\ref{thm:EI IE crit}. The fact that $\VDFG$ verifies the hypotheses of Theorem\,\ref{thm:dens def} is proved in the proof of Corollary\,\ref{cor:VDF dens def}.
\end{proof}

\appendix
\part*{Appendix}

\section{Uniform stable embeddedness of Henselian valued fields}
\label{sec:unif stab emb}

The goal of this section is to study stable embeddedness in pairs of valued fields and, in particular, to show that there exist models of $\ACVF$ uniformly stably embedded in every elementary extension. These models are used to prove that there are models of $\VDF$ whose underlying valued field is stably embedded in every elementary extension in the proof of Theorem\,\ref{thm:VDF}. These results are valid in any characteristic.

Following Baur, let us first introduce the notion of a separated pair of valued fields.

\begin{definition}[sep pair](Separated pair)
Let $K\subseteq L$ be an extension of valued fields. Call a tuple $a\in L$ $K$-separated if for any tuple $\lambda\in K$, $\val(\sum_{i}\lambda_{i}a_{i}) = \min_{i}\{\val(\lambda_{i}a_{i})\}$. The pair $K\subseteq L$ is said to be separated if any finite dimensional sub-$K$-vector space of $L$ has a $K$-separated basis.
\end{definition}

Recall that a maximally complete field is a field where every chain of balls has a point. Let us now recall a well known result of \cite{Bau-Sep}.

\begin{proposition}[SCimpSep]
If $K$ is maximally complete, any extension $K\subseteq L$ is separated.
\end{proposition}

Following \cite{CubDel,Del-ExtQp}, let us give the links between separation of the pair $K\subseteq L$ and uniform stable embeddedness of $K$ in $L$. But first let us define this last notion.

\begin{definition}(Uniform stable embeddedness)
Let $M$ be an $\LL$-structure and $A\subseteq M$. We say that $A$ is uniformly stably embedded if for all formulas $\phi(x;t)$ there exists a formula $\chi(x;s)$ such that for all tuples $b\in M$ there exists a tuple $a\in A$ such that $\phi(A,b) = \chi(A,a)$.
\end{definition}

The proof of Proposition\,\ref{prop:unif field st emb} is taken almost word for word from the one in \cite{CubDel}, although we put more emphasis on uniformity here. Let $\LL$ denote the two sorted language for valued fields.

\begin{proposition}[unif field st emb]
Let $M\models\ACVF$ and $\phi(x;s)$ an $\LL$-formula where $x$ is a tuple of $\K$-variables. There exists an $\restr{\LL}{\Valgp}$-formula $\psi(y;u)$ and polynomials $Q_{i}\in\Zz[\uple{X},\uple{T}]$ such that for any $N\substr M$, where the pair $\K(N)\subseteq\K(M)$ is separated, and any $a\in M$, there exists $b\in \K(N)$ and $c\in \Valgp(M)$ such that $\phi(N;a) = \psi(\val(\uple{Q}(N,b));c)$.
\end{proposition}

\begin{proof}
By elimination of quantifiers (and the fact that $\K$ is dominant), we may assume that $\phi(x;a)$ is of the form $\psi(\val(\uple{P}(x)))$ where $\uple{P}$ is a tuple of polynomials from $\K(M)[\uple{X}]$, $n\in\Nn$ and $\psi$ is an $\restr{\LL}{\lt}$-formula. Let us write each $P_{i}$ as $\sum_{\mu} a_{i,\mu}\uple{X}^{\mu}$. As the pair $\K(N)\subseteq\K(M)$ is separated, the $\K(N)$-vector space generated by the $a_{i,\mu}$ is generated by a $\K(N)$-separated tuple $\uple{d}\in\K(M)$. Note that $\card{\uple{d}}\leq\card{\uple{a}}$ and adding zeros to $\uple{d}$ we may assume $\card{\uple{d}} = \card{\uple{a}}$. For each $i$ and $\mu$, find $\lambda_{i,\mu,j}\in\K(N)$ such that $a_{i,\mu} = \sum_{j}\lambda_{i,\mu,j}d_{j}$. We can rewrite each $P_{i}$ as $\sum_{j} d_{j} Q_{i,j}(\uple{X},\uple{\lambda})$, where $Q_{i,j}\in \Zz[\uple{X},\uple{T}]$ does not depend on $\uple{a}$. For all $x\in K(N)$ we have $\val(P_{i}(x)) = \min_{j}\{\val(d_{j}Q_{i,j}(x,\uple{\lambda}))\}$. The proposition now follow easily by taking $b = \uple{\lambda}$ and $c = \val(\uple{d})$.
\end{proof}

\begin{theorem}[AKEstembACVF]
Let $K\subseteq L$ be a separated pair of valued fields such that $L$ is algebraically closed. Then $K$ is stably embedded in $L$ if and only if $\Valgp(K)$ is stably embedded in $\Valgp(L)$, as an ordered Abelian group. Moreover, if $\Valgp(K)$ is uniformly stably embedded in $\Valgp(L)$, then $K$ is uniformly stably embedded in $L$.
\end{theorem}

\begin{proof}
This follows immediately from Proposition\,\ref{prop:unif field st emb}.
\end{proof}

\begin{remark}
The computations of Proposition\,\ref{prop:unif field st emb} also applies to the $\ltf$ map (and the higher order leading terms $\ltf[n] : \K \to \K/1+n\Mid = \lt[n]$ in the mixed characteristic case). We get that $\ltf[n](P_{i}(x)) = \sum_{j}\ltf[n](d_{j}Q_{i,j}(x,\uple{\lambda}))$.

It follows that if the pair $K\subseteq L$ is separated and $L$ is a characteristic zero Henselian field, $K$ is stably embedded in $L$ if and only if $\bigcup_n\lt[n](K)$ is stably embedded in $\bigcup_n\lt[n](L)$. If we add angular components (which corresponds to splittings of $\lt[n]$) and restrict to the unramified case (either residue characteristic zero or positive residue characteristic $p$ and $\val(p)$ is minimal positive), then $K$ is stably embedded in $L$ if and only of $\Valgp(K)$ is stably embedded in $\Valgp(L)$ and $\res(K)$ is stably embedded in $\res(L)$.
\end{remark}

\begin{corollary}[Hahn unif stab emb]
Let $k$ be any algebraically closed field. The Hahn field $K := k((t^{\Rr}))$ is uniformly stably embedded (as a valued field) in any elementary extension.
\end{corollary}

\begin{proof}
The field $K$ is Henselian, as are all Hahn fields. Its residue field $k$ is algebraically closed and its value group $\Rr$ is divisible. It follows that $K$ is algebraically closed. By Proposition\,\ref{prop:SCimpSep}, any extension $K\subseteq L$ is separated. By Theorem\,\ref{thm:AKEstembACVF}, it suffices to show that $\Rr$ is uniformly stably embedded (as an ordered group) in any elementary extension. But that follows from the fact that $(\Rr,<)$ is complete and $(\Rr,+,<)$ is $o$-minimal, see \cite[Corollary\,64]{CheSim-Ext2}.
\end{proof}

\begin{remark}
An easy consequence of this result is that the constant field $\cst{\K}$ is stably embedded in models of $\VDF$. Indeed by quantifier elimination, we only need to show that $\cst{\K}$ is stably embedded in $\K$ as a valued. But that follows from Corollary\,\ref{cor:Hahn unif stab emb} and the fact that for any $k\models\DCF[0]$, $K = k((t^{\Rr})) \models\VDF$ (for the derivation described in Example\,\ref{ex:Hahn VDF}) and its constant field $\cst{K} = \cst{k}((t^{\Rr}))$ is uniformly stably embedded in $K$.

It then follows from quantifier elimination that $\cst{\K}$ is a pure algebraically closed field.
\end{remark}


\printsymbols
\printbibli
\end{document}